\documentclass[a4paper,reqno]{amsart}
\usepackage{amssymb}
\usepackage[cp1250]{inputenc}
\usepackage[OT4]{fontenc}
\usepackage{enumerate}
\usepackage{mathrsfs}
\usepackage{xcolor}
\def\CC{\mathbb{C}}
\def\RR{\mathbb{R}}

\def\NN{\mathbb{N}}
\def\DD{\mathbb{D}}
\def\BB{\mathbb{B}}
\def\TT{\mathbb{T}}

\def\bs{\boldsymbol}
\def\lon{\longrightarrow}

\def\CDD{\overline{\DD}}
\def\wi{\widetilde}

\def\pa{\partial}
\def\ov{\overline}
\def\su{\subset}

\def\lan{\langle}
\def\ran{\rangle}
\def\sm{\setminus}
\def\la{\lambda}
\def\al{\alpha}
\def\OO{{\mathcal O}}
\def\cC{{\mathcal C}}
\def\ee{{\mathcal E}}

\def\sss{\mathcal S}
\def\psh{{\mathcal{PSH}}}
\def\eps{\varepsilon}
\DeclareMathOperator{\id}{id}

\DeclareMathOperator{\re}{Re}

\DeclareMathOperator{\im}{Im}

\DeclareMathOperator{\dist}{dist}
\DeclareMathOperator{\aut}{Aut}

\renewcommand{\phi}{\varphi}
\newtheoremstyle{prop}{}{}{\it}{}{\bf}{.}{ }{}
\newtheoremstyle{thm}{}{}{\it}{}{\bf}{.}{ }{}
\newtheoremstyle{lem}{}{}{\it}{}{\bf}{.}{ }{}
\newtheoremstyle{cor}{}{}{\it}{}{\bf}{.}{ }{}
\newtheoremstyle{rem}{}{}{\normalfont}{}{\bf}{.}{ }{}
\theoremstyle{prop}
\newtheorem{prop}{Proposition}[section]
\theoremstyle{thm}
\newtheorem{thm}[prop]{Theorem}
\theoremstyle{lem}
\newtheorem{lem}[prop]{Lemma}
\theoremstyle{cor}
\newtheorem{cor}[prop]{Corollary}
\theoremstyle{rem}
\newtheorem{rem}[prop]{Remark}
\numberwithin{equation}{section}
\begin{document}
\title[(Weak) $m$-extremals and $m$-geodesics]{(Weak) $\bs m$-extremals and $\bs m$-geodesics}
\author{Tomasz Warszawski}
\thanks{The work is supported by the grant of the Polish National Science Centre no. DEC-2012/05/N/ST1/03067}
\email{warszawski.tomasz@gmail.com}
\begin{abstract}
We present a collection of results on (weak) $m$-extremals and $m$-geodesics, concerning general properties, the planar case, quasi-balanced pseudoconvex domains, complex ellipsoids, the Euclidean ball and boundary properties. We prove 3-geodesity of 3-extremals in the Euclidean ball. Equivalence of weak $m$-extremality and $m$-extremality in some class of convex complex ellipsoids, containing symmetric ones and $\mathcal C^2$-smooth ones is showed. Moreover, first examples of 3-extremals being not 3-geodesics in convex domains are given.
\end{abstract}
\maketitle
\tableofcontents
\section{Introduction}
\subsection{Idea of (weak) $\bs m$-extremals and $\bs m$-geodesics}
This paper may be treated as a continuation of \cite{kz}, where these objects were investigated from the point of view of geometric function theory. The notion of $m$-extremals comes from \cite{aly}~(cf.~\cite{aly1}) and was used to studying interpolation problems in the symmetrised bidisc --- a~special domain appearing in what is known as $\mu$-synthesis. It is a kind of approach to the spectral Nevanlinna-Pick problem (see also \cite{kos}), in which domains like the tetrablock and the pentablock occur naturally. They have been intensively studied of late in geometric function theory. However, $m$-extremals are in some sense too restricted. Therefore, it was natural to define \emph{weak} $m$-extremals; on the other side, a~stronger notion of $m$-geodesics let us produce $m$-extremals efficiently (\L.~Kosi\'nski and W.~Zwonek introduced both notions).

G. Pick \cite{pic} was the first who observed that Blaschke products have some extremal property in the unit disc $\DD$. The result formulated in our language claims that a holomorphic function $f:\DD\lon\DD$ is a (weak) $m$-extremal if and only if it is a~non-constant Blaschke product of degree at most $m-1$. More famous Pick (or Nevanlinna-Pick) theorem \cite{mar, nev, nev1} describes situations, in which a given interpolation problem in $\DD$ has a~solution. These results were obtained by the Schur's reduction \cite{gar, sch}.

A more general view on extremal problems (using special functionals) was presented in \cite{pol}. A. Edigarian developed these ideas in the crucial work \cite{e}, where among others the necessary form of weak $m$-extremals in complex ellipsoids is given. We will strongly use that result. A related problem with infinitely many interpolation data was studied in~\cite{at}.

There is a significant relationship between discussed objects and the theory of holomorphically contractible functions \cite{jp1, jp, kob} --- weak $m$-extremals (resp. $m$-geodesics) generalize classical Lempert extremals (resp. geodesics).

\subsection{Notation and definitions}
In what follows and if not mentioned otherwise, we assume that $m\ge 2$ is natural. Let $D\su\CC^n$ be a domain. Denote by $\OO(\CDD,D)$ the set of mappings holomorphic in a neighborhood of $\CDD$ with values in $D$. Moreover, let $\la_1,\ldots,\la_m\in\DD$ be distinct points (distinct = pairwise distinct).

A holomorphic mapping $f:\DD\lon D$ is called a \emph{weak $m$-extremal} for $\la_1,\ldots,\la_m$ if there is no map $h\in\OO(\CDD,D)$ such that $h(\la_j)=f(\la_j)$, $j=1,\ldots,m$. Naturally, weak $m$-extremality means weak $m$-extremality for some $\la_1,\ldots,\la_m$.

If the above condition is satisfied for any different numbers $\la_1,\ldots,\la_m\in\DD$, we say that $f$ is an \emph{$m$-extremal}.

Note that a map $f\in\OO(\DD,D)$ is a weak $m$-extremal for $\la_1,\ldots,\la_m$ if and only if there is no $g\in\OO(\DD,D)$ with $g(\la_j)=f(\la_j)$, $j=1,\ldots,m$, and $g(\DD)\su\su D$ (cf. Lemma \ref{28}$(a)$).

For $\alpha\in\DD$ define the \emph{M\"obius function} $$\label{mob}m_\alpha(\la):=\frac{\la-\alpha}{1-\ov\alpha\la},\quad\la\in\DD.$$ We shall consider finite \emph{Blaschke products}, that is functions $$B:=\zeta\prod_{j=1}^km_{\al_j},$$ where $k\in\NN_0$, $\al_j\in\DD$, $\zeta\in\TT:=\pa\DD$ (we assume that $0\notin\NN$ and $\NN_0:=\NN\cup\{0\}$). The number $k$ is said to be a \emph{degree} and is denoted by $\deg B$. In case $k=0$, the function $B$ is a unimodular constant $\zeta$.

A holomorphic mapping $f:\DD\lon D$ is said to be an \emph{$m$-geodesic} if there exists $F\in\OO(D,\DD)$ such that $F\circ f$ is a non-constant Blaschke product of degree at most $m-1$. We call such $F$ an \emph{$m$-left inverse}.

Note that a holomorphic map is a weak 2-extremal (resp. a 2-geodesic) if and only if it is a Lempert extremal (resp. a geodesic). Recall that a mapping $f\in\OO(\DD,D)$ is a \emph{Lempert extremal} if $\bs\ell_D(f(\la_1),f(\la_2))=\bs p(\la_1,\la_2)$ for some different $\la_1,\la_2\in\DD$, where $\bs p$ stands for the Poincar\'e distance on $\DD$ and $$\bs\ell_{D}(z,w):=\inf\{\bs p(\la_1,\la_2):\la_1,\la_2\in\DD\textnormal{ and }\exists f\in\mathcal{O}(\mathbb{D},D):f(\la_1)=z,\ f(\la_2)=w\}$$ is the \emph{Lempert function} of $D$. We call $f$ a \emph{geodesic} if $\bs c_D(f(\la_1),f(\la_2))=\bs p(\la_1,\la_2)$ for any (equivalently for some different) $\la_1,\la_2\in\DD$, where $$\bs c_D(z,w):=\sup\{\bs p(F(z),F(w)):F\in\mathcal{O}(D,\DD)\}$$ is the \emph{Carath\'eodory pseudodistance} of $D$. This is exactly the case, when $f$ has a~2-left inverse.

From the description of $m$-extremals in $\DD$ it follows that in any domain $m$-geodesity implies $m$-extremality. It is obvious that for all considered notions the `level' $m$ implies $m+1$. They are invariant under biholomorphisms and compositions with automorphisms of $\DD$.

\subsection{Main results}
It is known from \cite{kz} that in the Euclidean ball we have $m$-extremals being not $m$-geodesics for $m\ge 4$. The missing case is solved in Theorem \ref{32}: \textbf{any 3-extremal of $\BB_n$ is a 3-geodesic}.

By the Lempert theorem \cite{Lem2, Lem3} (cf. \cite[Chapter 11]{jp} and \cite{rw}), any weak 2-extremal of a convex domain is a 2-geodesic, in particular a 2-extremal. Thus the following question about a `weak' generalization of this result seems to be~important: whether a weak $m$-extremal, $m\ge 3$, of a convex domain has to be an $m$-extremal. We do not know it, however we have found \textbf{convex domains, in which for any $\bs{m\ge 3}$ there exists an $\bs m$-extremal being not an $\bs m$-geodesic} (first convex examples for $m=3$). One of them is the complex ellipsoid $\ee(1/2,1/2)$, where $\ee(p):=\{z\in\CC^n:|z_1|^{2p_1}+\ldots+|z_n|^{2p_n}<1\}$ (Proposition \ref 1). Another example follows from Proposition \ref{ab}.

The next results, we would like to draw attention to, are Propositions \ref{31} and \ref{30}. We define some family $\ee(p)$, $p\in\sss_n$, which contains all symmetric convex and all $\cC^2$-smooth complex ellipsoids. It turns out that \textbf{in $\bs{\ee(p)}$ such that $\bs{p\in\sss_n}$, weak $\bs m$-extremality equals $\bs m$-extremality}. We get moreover some $l$-extremality of all weak $m$-extremals for $p$ such that $p_j/q_j\in\NN$, $j=1,\ldots,n$, where $q\in\sss_n$ ($l$~is bounded by a~function of $m$ and $p$).

We also deal with dividing $m$-geodesics of quasi-balanced pseudoconvex domains by the identity function on the unit disc. The aim is to decide whether the new map is an $(m-1)$-geodesic. The reasoning used in the proof of \cite[Theorem 3]{ekz} gives a positive answer for $m=3$. Most interesting is the balanced case, in which we give convex counterexamples for $m\ge 4$ (Corollary \ref{a} and Propositions \ref{ab}, \ref{40}).

Some of the results answer partially to questions posed at the end. Their occurrences in the text are marked (P$n$).

We have already two general questions: whether it is possible to find a 2-extremal being not a 2-geodesic (P\ref{2e2g}), and whether there exists an $m$-extremal, which is not any $k$-geodesic (P\ref{mk}).

\section{General properties and the planar case}
Denote by $|\cdot|$ the Euclidean norm and $\|f\|_S:=\sup_S|f|$.

\begin{lem}\label{28}
Let $D\su\CC^n$ be a domain and let $\la_1,\ldots,\la_m\in\DD$ be different points.
\begin{enumerate}[$(a)$]
\item Fix $z_1,\ldots,z_n\in D$. Then there exists $h\in\OO(\CDD,D)$ such that $h(\la_j)=z_j$, $j=1,\ldots,m$, if and only if there exists $g\in\OO(\DD,D)$ such that $g(\la_j)=z_j$, $j=1,\ldots,m$, and $g(\DD)\su\su D$.
\item\label{29} Let $\DD\ni\la_j^{(k)}\to\la_j$, $k\to\infty$, and let $f:\DD\lon D$ be a weak $m$-extremal for $\la_1^{(k)},\ldots,\la_m^{(k)}$. Then $f$ is a weak $m$-extremal for $\la_1,\ldots,\la_m$.
\item Assume that $f_k,f\in\OO(\DD,D)$, $f_k(\la_j)\to f(\la_j)$, $k\to\infty$, and $f_k$ are weak $m$-extremals for $\la_1,\ldots,\la_m$. Then $f$ is, as well.
\item If $\OO(\DD,D)\ni f_k\to f\in\OO(\DD,D)$ pointwise and any $f_k$ is an $m$-extremal, then $f$ is, as well.
\end{enumerate}
\end{lem}

\begin{proof}
Let $w_1,\ldots,w_m\in\CC^n$, $w:=(w_1,\ldots,w_m)$. The polynomial mapping $$P_w(\la):=\sum_{j=1}^m\prod_{k\neq j}\frac{\la-\la_k}{\la_j-\la_k}w_j,\quad\la\in\CC,$$ has the property that $P_w(\la_l)=w_l$, $l=1,\ldots,m$, and $\|P_w\|_S\to 0$ if $w\to 0$, for any $\emptyset\neq S\su\su\CC$.

$(a)$ If we have $g$, then consider $g_r(\la):=g(\la/r)$, $\la\in r\DD$, $r>1$. As $$g_r(\la_j)+g(\la_j)-g_r(\la_j)=z_j,$$ we put $w_j=w_j(r):=g(\la_j)-g_r(\la_j)$ and $h:=g_r+P_{w(r)}$ for $r$ close enough to 1.

$(b)$ Suppose that there exists $h\in\OO(\DD,D)$ such that $h(\la_j)=f(\la_j)$, $j=1,\ldots,m$, and $h(\DD)\su\su D$. We proceed similarly as above with the equation $$h(\la_j^{(k)})+h(\la_j)-h(\la_j^{(k)})+f(\la_j^{(k)})-f(\la_j)=f(\la_j^{(k)})$$ and get a contradiction with weak $m$-extremality of $f$ for $\la_1^{(k)},\ldots,\la_m^{(k)}$ if $k>>1$.

$(c)$ If there were exist $h\in\OO(\DD,D)$ with $h(\la_j)=f(\la_j)$, $j=1,\ldots,m$, and $h(\DD)\su\su D$, we would have $$h(\la_j)+f_k(\la_j)-f(\la_j)=f_k(\la_j),$$ whence for big $k$ the map $f_k$ would be not a weak $m$-extremal for $\la_1,\ldots,\la_m$.

$(d)$ It follows from $(c)$.
\end{proof}

From the definition of a weak $m$-extremal follows

\begin{lem}\label{cart}
Let $D_j\su\CC^{k_j}$ be domains and let $f_j\in\OO(\DD,D_j)$, $j=1,\ldots,n$. Then the mapping $(f_1,\ldots,f_n):\DD\lon D_1\times\ldots\times D_n$ is a weak $m$-extremal for $\la_1,\ldots,\la_m$ if and only if at least one of the maps $f_j$ is a weak $m$-extremal for $\la_1,\ldots,\la_m$.
\end{lem}

In particular, we have the following description for the polydisc $\DD^n$.

\begin{rem}\label{pol}
Let $f:\DD\lon\DD^n$ be a holomorphic mapping. Then the following conditions are equivalent
\begin{enumerate}[$(a)$]
\item $f$ is a weak $m$-extremal,
\item $f$ is an $m$-extremal,
\item $f$ is an $m$-geodesic,
\item $f_j$ is a non-constant Blaschke product of degree $\leq m-1$ for some $j\in\{1,\ldots,n\}$.
\end{enumerate}
\end{rem}

\bigskip

As already mentioned, a holomorphic function $f:\DD\lon\DD$ is a (weak) $m$-extremal if and only if it is a non-constant Blaschke product of degree $\leq m-1$. Thus weak $m$-extremality coincides with $m$-extremality and $m$-geodesity and is entirely described in all simply connected proper domains in $\CC$.

Polynomial interpolation shows immediately that $\CC^n$, $(\CC_*)^k$ and  $\CC^n\times(\CC_*)^k$ do not have weak $m$-extremals.

We present a description of weak $m$-extremals of remaining planar domains, that is domains $D\su\CC$ such that $\#(\CC\sm D)\ge 2$ and $D$ is not biholomorphic to $\DD$. These are all non-simply connected taut domains on the plane. We start with the following

\begin{lem}\label{42}
Let\, $\Pi:\wi D\lon D$ be a holomorphic covering between domains $\wi D,D\su\CC^n$. Assume that $\wi f:\DD\lon\wi D$ is an $m$-extremal. Then $f:=\Pi\circ\wi f:\DD\lon D$ is a~weak $m$-extremal. 
\end{lem}

\begin{proof}
Suppose that $f$ is not a weak $m$-extremal. Then for any $k\geq m$ there exist $r_k>1$ and a function $h_k\in\OO(r_k\DD,D)$ with $h_k(j/k)=f(j/k)$, $j=0,\ldots,m-1$. Since $\wi f(0)\in\Pi^{-1}(\{h_k(0)\})$, we may lift $h_k$ by $\Pi$ to $\wi h_k\in\OO(r_k\DD,\wi D)$ with the condition $\wi h_k(0)=\wi f(0)$. By the Montel theorem, some subsequence $\wi h_{l_k}$ is locally uniformly convergent on $\DD$. Then for big $k$ all the points $\wi h_{l_k}(j/l_k)$, $\wi f(j/l_k)$ ($j=0,\ldots,m-1$) drop into a neighborhood of $\wi f(0)$, on which $\Pi$ is biholomorphic. From $$\Pi(\wi h_{l_k}(j/l_k))=h_{l_k}(j/l_k)=f(j/l_k)=\Pi(\wi f(j/l_k))$$ we infer that $\wi h_{l_k}(j/l_k)=\wi f(j/l_k)$, $j=0,\ldots,m-1$, which contradicts $m$-extremality of $\wi f$.
\end{proof} 

\begin{prop}\label{27}
Let $D\su\CC$ be a non-simply connected taut domain and let\, $\Pi:\DD\lon D$ be a~holomorphic covering. Then a holomorphic function $f:\DD\lon D$ is a weak $m$-extremal if and only if $f=\Pi\circ B$, where $B$ is a non-constant Blaschke product of degree $\leq m-1$. Moreover, $f$ is not an $m$-extremal.
\end{prop}

\begin{proof}
Any holomorphic function $f:\DD\lon D$ can be lifted by $\Pi$, i.e. there exists $B\in\OO(\DD,\DD)$ with $f=\Pi\circ B$. Assume that $f$ is a weak $m$-extremal for $\la_1,\ldots,\la_m$. Then $B$ is a weak $m$-extremal for $\la_1,\ldots,\la_m$, that is a non-constant Blaschke product of degree $\leq m-1$.

Reversely, assume that $f=\Pi\circ B$, where $B$ is a non-constant Blaschke product of degree $\leq m-1$. By Lemma \ref{42}, the function $f$ is a weak $m$-extremal.

Suppose that $f$ is an $m$-extremal. We claim that $\Pi$ is an $m$-extremal. Indeed, suppose contrary. It follows that there exist distinct points $\la_1,\ldots,\la_m\in\DD$ and a~holomorphic function $h:\DD\lon D$ with $h(\la_j)=\Pi(\la_j)$, $j=1,\ldots,m$, and $h(\DD)\su\su D$. Let $\mu_j\in\DD$ be such that $\la_j=B(\mu_j)$. Then $h\circ B$ gives a contradiction with weak $m$-extremality of $f$ for $\mu_1,\ldots,\mu_m$. Since $\Pi$ is of infinite (countable) multiplicity, for any $a\in D$ the set $\Pi^{-1}(\{a\})\su\DD$ is infinite. Therefore, the constant function $a$~interpolates $\Pi$ for any different numbers $\la_1,\ldots,\la_m\in\Pi^{-1}(\{a\})$, contradiction.
\end{proof} 

\section{Quasi-balanced pseudoconvex domains}
Given $k=(k_1,\ldots,k_n)\in(\NN_0^n)_*$. A domain $D\su\CC^n$ such that $$\la\in\CDD,\ z\in D\Longrightarrow(\la^{k_1}z_1,\ldots,\la^{k_n}z_n)\in D,$$ is called $k$-\emph{balanced} or generally \emph{quasi-balanced}. A $(1,\ldots,1)$-balanced domain is \emph{balanced}.

\begin{lem}\label{fi}
Let $D\su\CC^n$ be a $k$-balanced pseudoconvex domain and let $f\in\OO(\DD,D)$ $($resp. $f\in\OO(\CDD,D)${$)$}. Assume that $$f=(m_\al^{k_1}\phi_1,\ldots,m_\al^{k_n}\phi_n),$$ where $\phi_j\in\OO(\DD)$ $($resp. $\phi_j\in\OO(\CDD)${$)$}, $\al\in\DD$ and $\phi:=(\phi_1,\ldots,\phi_n)$. Then either $\phi(\DD)\su D$ or $\phi(\DD)\su\pa D$ $($resp. $\phi\in\OO(\CDD,D)${$)$}.
\end{lem}

\begin{proof}
Consider two cases.

$(a)$ $k_1,\ldots,k_n\ge 1$. Let $$\label{min}h(z):=\inf\left\{t>0:\left(\frac{z_1}{t^{k_1}},\ldots,\frac{z_n}{t^{k_n}}\right)\in D\right\},\quad z\in\CC^n,$$ stand for the \emph{$k$-Minkowski function} of the domain $D$. If $k=(1,\ldots,1)$, we have the classical \emph{Minkowski function}. Then
\begin{itemize}
\item $D=\{z\in\CC^n:h(z)<1\}$,
\item $h(\la^{k_1}z_1,\ldots,\la^{k_n}z_n)=|\la|h(z)$, $z\in\CC^n$, $\la\in\CC$,
\item $D$ is pseudoconvex if and only if $\log h\in\psh(\CC^n)$ \cite{n} (cf. \cite[Proposition 2.2.15]{jp}).
\end{itemize}

We have $$\limsup_{\la\to\TT}h(\phi(\la))=\limsup_{\la\to\TT}h(f(\la))\le 1$$  $$(\text{resp. }h\circ\phi=h\circ f<1\text{ on }\TT),$$ whence either $\phi(\DD)\su D$ or $\phi(\DD)\su\pa D$ (resp. $\phi\in\OO(\CDD,D)$).

$(b)$ In the opposite case assume that $k_1=\ldots=k_s=0$, $k_{s+1},\ldots,k_n\ge 1$, where $1\le s\le n-1$. Denote $z':=(z_1,\ldots,z_s)$, $z'':=(z_{s+1},\ldots,z_n)$ for $z\in\CC^n$. Let $G$ be the projection of $D$ on $\CC^s$. Define $h$ by the same formula as before, but for $z\in G\times\CC^{n-s}$. Further we proceed analogously as in \cite{n} (cf. \cite[Proposition 2.2.15]{jp}). We define the map $\Phi:G\times\CC^{n-s}\lon\CC^n$ as $\Phi(z):=(z',z_{s+1}^{k_{s+1}},\ldots,z_n^{k_n})$ and put $\wi D:=\Phi^{-1}(D)$, $\wi h:=h\circ\Phi$. Then $\wi h(z',\la z'')=|\la|\wi h(z',z'')$, which means that $$\wi D=\{(z',z'')\in G\times\CC^{n-s}:\wi h(z',z'')<1\}$$ is a pseudoconvex Hartogs domain over $G$ with balanced fibers. For any point $z'\in G$, the function $\wi h(z',\cdot)$ is the Minkowski function of the fiber $$\wi D_{z'}:=\{z''\in\CC^{n-s}:(z',z'')\in\wi D\},$$ hence $G$ is pseudoconvex and $\log\wi h\in\psh(G\times\CC^{n-s})$ \cite[Proposition 4.1.14]{JJ}. As $h(z)=\wi h(z',\sqrt[k_{s+1}]{z_{s+1}},\ldots,\sqrt[k_n]{z_n})$ (with an arbitrary choice of the roots), we have $\log h\in\psh(G\times(\CC_*)^{n-s})$. From the removable singularities theorem it follows that $\log h\in\psh(G\times\CC^{n-s})$.

We finish the proof as in $(a)$.
\end{proof}

The following lemma will be crucial in the study of (weak) $m$-extremals e.g. in complex ellipsoids.

\begin{lem}\label{10}
Let $D\su\CC^n$ be a $k$-balanced pseudoconvex domain and let $f:\DD\lon D$ be a weak $m$-extremal for $\la_1,\ldots,\la_m$.
\begin{enumerate}[$(a)$]
\item Assume that $f=(m_\al^{k_1}\phi_1,\ldots,m_\al^{k_n}\phi_n)$, $\phi_j\in\OO(\DD)$, $\al\in\DD$, $\phi:=(\phi_1,\ldots,\phi_n)$. Then either $\phi(\DD)\su D$ or $\phi(\DD)\su\pa D$, and in the first case
\begin{enumerate}[$(i)$]
\item $\phi$ is a weak $m$-extremal for $\la_1,\ldots,\la_m$.
\item if $m\geq 3$, $\la_m=\al$ and $k_1,\ldots,k_n\ge 1$, then $\phi$ is a weak $(m-1)$-extremal for $\la_1,\ldots,\la_{m-1}$.
\end{enumerate}
\item Suppose that $\DD\ni\la_{m+1}\ne\la_1,\ldots,\la_m$ and $k_1,\ldots,k_n\leq 1$, $l\in\NN$. Then the map $\psi_{(l)}:=(m_{\la_{m+1}}^{lk_1}f_1,\ldots,m_{\la_{m+1}}^{lk_n}f_n):\DD\lon D$ is a weak $(m+1)$-extremal for $\la_1,\ldots,\la_{m+1}$.
\end{enumerate}
\end{lem}

\begin{proof}
$(a)$ Assume that $\phi(\DD)\su D$.

\begin{enumerate}[$(i)$]
\item Suppose that there exists $h\in\OO(\CDD,D)$ with $h(\la_j)=\phi(\la_j)$, $j=1,\ldots,m$. Then $g:=(m_\al^{k_1}h_1,\ldots,m_\al^{k_n}h_n)\in\OO(\CDD,D)$ satisfies $g(\la_j)=f(\la_j)$, $j=1,\ldots,m$, contradiction.
\item Assume that there is $h\in\OO(\CDD,D)$ with $h(\la_j)=\phi(\la_j)$, $j=1,\ldots,m-1$. Then $g:=(m_\al^{k_1}h_1,\ldots,m_\al^{k_n}h_n)\in\OO(\CDD,D)$ satisfies $g(\la_j)=f(\la_j)$, $j=1,\ldots,m-1$, and $g(\al)=f(\al)=0$, contradiction.
\end{enumerate}

$(b)$ We proceed inductively on $l$. For $l=1$ assume the existence of a mapping $h\in\OO(\CDD,D)$ such that $h(\la_j)=\psi_{(1)}(\la_j)$, $j=1,\ldots,m+1$. The mapping $g:=(h_1/m_{\la_{m+1}}^{k_1},\ldots,h_n/m_{\la_{m+1}}^{k_n})\in\OO(\CDD,D)$ satisfies $g(\la_j)=f(\la_j)$, $j=1,\ldots,m+1$, contradiction.

Step $l\Longrightarrow l+1$: proceed as above for $\psi_{(l)}$ and $\psi_{(l+1)}$ instead of $f$ and $\psi_{(1)}$ respectively.
\end{proof}

Assuming that $f$ is an $m$-extremal, it seems that generally $\psi_{(1)}$ should not be an $(m+1)$-extremal (P\ref{m+1}).

Lemmas \ref{10}$(a)(ii)$ and \ref{28}$(b)$ give

\begin{cor}\label{17}
Let $D\su\CC^n$ be a $k$-balanced pseudoconvex domain and let $f:\DD\lon D$ be an $m$-extremal. Assume that $f=(m_\al^{k_1}\phi_1,\ldots,m_\al^{k_n}\phi_n)$, $\phi_j\in\OO(\DD)$, $\al\in\DD$, $m\geq 3$, $k_1,\ldots,k_n\ge 1$. Then either $\phi:=(\phi_1,\ldots,\phi_n)$ is an $(m-1)$-extremal of $D$ or $\phi(\DD)\su\pa D$.
\end{cor}

The question arises, whether the analogue of Corollary \ref{17} holds for $m$-geodesics.

\begin{rem}\label{ana}
We shall show that it is false even in the convex case for
\begin{enumerate}[$(a)$]
\item $m\ge 4$ and some $k\neq(1,\ldots,1)$ (Corollary \ref{a}).
\item $m\ge 4$ and $k=(1,\ldots,1)$ (Proposition \ref{ab}).
\item $m\ge 5$ and $k=(1,\ldots,1)$ in some complex ellipsoid (Proposition \ref{40}).
\end{enumerate}
\end{rem}

It would be interesting to decide what happens for $m=4$ and $k=(1,\ldots,1)$ in (not necessarily convex) complex ellipsoids (P\ref{37}).

It turns out the answer for $m=3$ is positive.

\begin{prop}\label{39}
Let $D\su\CC^n$ be a $k$-balanced pseudoconvex domain and let $f:\DD\lon D$ be a $3$-geodesic. Assume that $f=(m_\al^{k_1}\phi_1,\ldots,m_\al^{k_n}\phi_n)$, $\phi_j\in\OO(\DD)$, $\al\in\DD$, $k_1,\ldots,k_n\ge 1$. Then either $\phi:=(\phi_1,\ldots,\phi_n)$ is a $2$-geodesic of $D$ or $\phi(\DD)\su\pa D$.

If $f$ is additionally a $2$-geodesic, then $\phi(\DD)\su\pa D$.
\end{prop}

Before showing it, recall

\begin{thm}[\cite{ekz}, Theorem 3]
Let $D\su\CC^n$ be a $k$-balanced pseudoconvex domain and let $f:\DD\lon D$ be a $2$-geodesic. Assume that $f=(m_\al^{k_1}\phi_1,\ldots,m_\al^{k_n}\phi_n)$, $\phi_j\in\OO(\DD)$, $\al\in\DD$. Then either $\phi:=(\phi_1,\ldots,\phi_n)$ is a $2$-geodesic of $D$ or $\phi(\DD)\su\pa D$.
\end{thm}

We will proceed very similarly as in that proof.

\begin{proof}[Proof of Proposition \ref{39}]
We can assume that $\al=0$, so $f(0)=0$. We know from Lemma \ref{fi} that either $\phi\in\OO(\DD,D)$ or $\phi\in\OO(\DD,\partial D)$. Suppose that the first case holds. Let $F\in\OO(D,\DD)$ be such that $F\circ f$ is a Blaschke product of degree 1 or 2. One may assume that $F(0)=0$, thus either $$F(f(\la))=\la,\quad\text{then denote }m:=1,$$ or $$F(f(\la))=\la m_\gamma(\la)\text{ for some }\gamma\in\DD,\quad\text{then }m:=m_\gamma.$$

Fix $z\in D$ and consider holomorphic functions defined on a neighborhood of $\CDD$ $$g_z:\la\longmapsto F(\la^{k_1}z_1,\ldots,\la^{k_n}z_n)/\la,\quad m:\lambda\longmapsto m(\lambda).$$ Since $|g_z(\la)|<1=|m(\la)|$ for $\lambda\in\TT$, the Rouch\'e theorem implies that the function $\DD\owns\lambda\longmapsto g_z(\la)-m(\la)\in\CC$ has in $\DD$ the same number of zeros as $m$. 

Therefore, it has no zeros if $m=1$. This fails for $z\in\phi(\DD)$, so the assumption $\phi(\DD)\su D$ is false in that case. The `additionally' claim is proved.

If $m=m_\gamma$, then the function $g_z(\la)-m(\la)$ has in $\DD$ exactly one root $G(z)$. Since the graph of the function $G:D\lon\DD$, equal to $$\{(z,\lambda)\in D\times\DD:F(\la^{k_1}z_1,\ldots,\la^{k_n}z_n)=\la m_\gamma(\la)\}$$ is an analytic set, we get that $G$ is holomorphic (\cite[Chapter V, \S 1]{l}, cf. \cite{chi} and \cite[Sekcja 5.5]{jjp}). Moreover, it follows from the definition that $G(\phi(\lambda))=\lambda$ for $\lambda\in\DD$, which finishes the proof.
\end{proof}

We finish the section with the following property.

\begin{lem}\label{11}
Let $D\su\CC^n$ be a $k$-balanced pseudoconvex domain and let $\phi\in\OO(\DD,\pa D)$, $\al\in\DD$, $k_1,\ldots,k_n\leq 1$. Then $f:=(m_\al^{k_1}\phi_1,\ldots,m_\al^{k_n}\phi_n):\DD\lon D$ is a weak $2$-extremal for $\al$ and $\mu\in\DD\sm\{\al\}$.

In particular, in the balanced case for any $a\in\pa D$ the map $\DD\ni\la\longmapsto\la a\in D$ is a weak $2$-extremal for $0$ and $\mu\in\DD_*$.
\end{lem}

\begin{proof}
One can assume that $\al=0$ and $f(\la)=(\la\psi(\la),\wi\psi(\la))$, $\la\in\DD$, where $\psi=(\phi_1,\ldots,\phi_s)$, $\wi\psi=(\phi_{s+1},\ldots,\phi_n)$ for some $1\le s\le n$. Suppose that $h\in\OO(\CDD,D)$ satisfies $h(0)=(0,\wi\psi(0))$ and $h(\mu)=(\mu\psi(\mu),\wi\psi(\mu))$. Then $h(\la)=(\la g(\la),\wi g(\la))$, $\la\in\DD$, for some map $(g,\wi g)\in\OO(\CDD,D)$. This contradicts the equality $(g,\wi g)(\mu)=(\psi(\mu),\wi\psi(\mu))=\phi(\mu)\in\pa D$.
\end{proof}

\section{Complex ellipsoids}
Let $p=(p_1,\ldots,p_n)\in\RR_{>0}^n$. The domain $$\ee(p):=\{z\in\CC^n:|z_1|^{2p_1}+\ldots+|z_n|^{2p_n}<1\}$$ is said to be a \emph{complex ellipsoid}. Write moreover $$\ee(\underbrace{p_0,\ldots,p_0}_n)=:\ee(p_0)\su\CC^n,\quad p_0>0,$$ for \emph{symmetric} complex ellipsoids. The unit Euclidean ball, shortly the \emph{ball}, is clearly $$\BB_n:=\ee(1)\su\CC^n.$$

\begin{rem}
\begin{enumerate}[$(a)$]
\item $\ee(p)$ is $k$-balanced and pseudoconvex, $k\in(\NN_0^n)_*$.
\item If $n\ge 2$, then $\ee(p)$ is convex if and only if $p_1,\ldots,p_n\ge 1/2$.
\item If $n\ge 2$, then $\ee(p)$ is $\cC^2$-smooth if and only if $p_1,\ldots,p_n\ge 1$.
\end{enumerate}
\end{rem}

In \cite[Proposition 11]{kz} are given $m$-extremals being not $m$-geodesics for $m\ge 4$ in $\BB_n$, $n\ge 2$. In Section \ref{23} we show that it is not possible in the ball for $m=3$. Below we have in particular 3-extremals, which are not 3-geodesics in a convex domain (Proposition \ref{ab} delivers other ones).

\begin{prop}\label 1
Let $m\ge 3$ and $0<a<1$. Then the map $$f(\la):=(a\la^{m-2},(1-a)\la^{m-1}),\quad\la\in\DD,$$ is an $m$-extremal, but not an $m$-geodesic of\, $\ee(1/2)\su\CC^2$.
\end{prop}

\begin{proof}
The mapping $$\DD\ni\la\longmapsto(a\la^{m-1},(1-a)\la^{m-1})\in\ee(1/2)$$ is an $m$-geodesic (the $m$-left inverse $z\longmapsto z_1+z_2$), so Lemma \ref{10}$(a)(i)$ says that $f$~is an $m$-extremal. Suppose that there exists a holomorphic function $F:\ee(1/2)\lon\DD$ such that $F\circ f$ is a~non-constant Blaschke product of degree $\leq m-1$. We can assume that $F(0)=0$, whence due to the Taylor expansion it follows that (with exactness up to a unimodular constant) either $F(f(\la))=\la^{m-2}$ or $F(f(\la))=\la^{m-2}m_\gamma(\la)$ for some $\gamma\in\DD$. 

In the first case we have $F(z)=z_1/a$, which is impossible. 

For the second case expand $F(z)=\alpha z_1+\beta z_2+\delta z_1^2+\ldots$ With fixed $z\in\ee(1/2)$, the function $g_z(\la):=F(\la z)/\la$, defined in a neighborhood of $\CDD$, is smaller than 1~in modulus on $\TT$. It follows that $g_z(0)=\alpha z_1+\beta z_2\in\DD$ for $z\in\ee(1/2)$. Hence $|\alpha|,|\beta|\leq 1$. From the comparison of the coefficients in the equation $$\la^{m-2}m_\gamma(\la)=F(f(\la)),\quad\la\in\DD,$$ we have \begin{align*}-\gamma&=\alpha a,\\1-|\gamma|^2&=\begin{cases}\beta(1-a)+\delta a^2,&m=3,\\\beta(1-a),&m\geq 4.\end{cases}\end{align*}

Consider first the possibility $m\ge 4$. We obtain \begin{equation}\label{44}1=|\alpha|^2a^2+\beta(1-a)\le a^2+1-a,\end{equation} contradiction.

For $m=3$ let the function $g:\DD\lon\CDD$ be given by $g(z_1):=F(z_1,0)/z_1=\alpha+\delta z_1+\ldots$

If $|\alpha|=1$, then $g$ is constant, in particular $\delta=0$. To get a contradiction use \eqref{44} (or note that $|\gamma|=a$, so $\beta=1+a$).

Otherwise $g$ has values in $\DD$. The function $h:=m_\alpha\circ g:\DD\lon\DD$ satisfies $h(0)=0$, hence $$1\geq|h'(0)|=\frac{|\delta|}{1-|\alpha|^2},\quad|\al|^2+|\delta|\leq 1.$$ This gives $$1=|\alpha|^2a^2+\delta a^2+\beta(1-a)\le a^2+1-a,$$ which finishes the proof.
\end{proof}

\begin{cor}[cf. Remark \ref{ana}$(a)$]\label{a}
Let $m\ge 4$ and $0<a<1$. Then the mapping $f:\DD\lon\ee(1/2)\su\CC^2$, $$f(\la):=(a\la^{m-1},(1-a)\la^{m-1}),\quad\la\in\DD,$$ is an $m$-geodesic such that $\phi(\la):=(f_1(\la)/\la^2,f_2(\la)/\la)$ is not an $(m-1)$-geodesic.
\end{cor}

\begin{prop}[cf. Remark \ref{ana}$(b)$]\label{ab}
Let $m\ge 4$ and let numbers $a,b>0$ be such that $4a^2+b=1$. Then the mapping $$f:\DD\lon D:=\{z\in\CC^3:(|z_1|+|z_2|)^2+|z_3|<1\},$$ $$f(\la):=(a\la,a\la^{m-2},b\la^{m-1}),$$ is an $m$-geodesic such that $\phi(\la):=f(\la)/\la$ is not an $(m-1)$-geodesic.
\end{prop}

\begin{proof}
The polynomial $4z_1z_2+z_3$ is an $m$-left inverse of $f$. Suppose that there is $F\in\OO(D,\DD)$ such that $$F(a,a\la^{m-3},b\la^{m-2})=B(\la),\quad\la\in\DD,$$ where $B$ is a non-constant Blaschke product of degree $\le m-2$. The function $$G:\{(z_2,z_3)\in\CC^2:(a+|z_2|)^2+|z_3|<1\}\ni(z_2,z_3)\longmapsto F(a,z_2,z_3)\in\DD$$ satisfies $$G(a\la^{m-3},b\la^{m-2})=B(\la),\quad\la\in\DD.$$ Assume that $G(0)=0$ and expand $G(z_2,z_3)=\alpha z_2+\beta z_3+\delta z_2^2+\ldots$ By considering the functions \begin{align*}g:\DD\ni z_2&\longmapsto G((1-a)z_2,0)/z_2\in\CDD,\\\DD\ni z_3&\longmapsto G(0,(1-a^2)z_3)/z_3\in\CDD,\end{align*} we get that $$|\al|(1-a)\le 1,\quad|\beta|(1-a^2)\le 1.$$ We can assume that either $B(\la)=\la^{m-3}$ or $B(\la)=\la^{m-3}m_\gamma(\la)$ for some $\gamma\in\DD$.

In the first case it follows that $\al a=1\ge|\al|(1-a)$, i.e. $2a\ge 1$. This is impossible. 

In the second one the following equations hold \begin{align*}-\gamma&=\alpha a,\\1-|\gamma|^2&=\begin{cases}\beta b+\delta a^2,&m=4,\\\beta b,&m\geq 5.\end{cases}\end{align*}

If $m\ge 5$, note that \begin{equation}\label{43}1=|\al|^2a^2+\beta b\le\frac{1}{(1-a)^2}a^2+\frac{1}{1-a^2}(1-4a^2),\end{equation} which reduces to $1\le 2a$, contradiction.

For $m=4$ let us come back to the function $$g(z_2)=\alpha(1-a)+\delta(1-a)^2z_2+\ldots$$ 

If $|\alpha(1-a)|=1$, then $g$ is constant, in particular $\delta=0$. Now use \eqref{43}.

Otherwise the function $h:=m_{\alpha(1-a)}\circ g:\DD\lon\DD$ satisfies $h(0)=0$. Hence $$1\geq|h'(0)|=\frac{|\delta|(1-a)^2}{1-|\alpha|^2(1-a)^2},\quad|\al|^2+|\delta|\leq\frac{1}{(1-a)^2}.$$ This gives $$1=|\al|^2a^2+\delta a^2+\beta b\le\frac{1}{(1-a)^2}a^2+\frac{1}{1-a^2}(1-4a^2),$$ i.e. $1\le 2a$.
\end{proof}

A description of 2-geodesics in $D$ (and in similar domains) may be found in \cite{v}.

\begin{prop}[cf. Remark \ref{ana}$(c)$]\label{40}
Let $m\ge 5$ and let positive numbers $a,b$ satisfy $2a^2+b=1$. Then the map $$f:\DD\lon\ee:=\{z\in\CC^3:|z_1|^2+|z_2|^2+|z_3|<1\},$$ $$f(\la):=(a\la,a\la^{m-2},b\la^{m-1}),$$ is an $m$-geodesic such that $\phi(\la):=f(\la)/\la$ is not an $(m-1)$-geodesic.
\end{prop}

\begin{proof}
The polynomial $2z_1z_2+z_3$ is an $m$-left inverse of $f$. Assume that there is a~holomorphic function $F:\ee\lon\DD$ such that $$F(a,a\la^{m-3},b\la^{m-2})=B(\la),\quad\la\in\DD,$$ where $B$ is a non-constant Blaschke product of degree at most $m-2$. Consider the function $$G:\{(z_2,z_3)\in\CC^2:|z_2|^2+|z_3|<1\}\ni(z_2,z_3)\longmapsto F(a,\sqrt{1-a^2}z_2,(1-a^2)z_3)\in\DD.$$ Then $$G(c\la^{m-3},d\la^{m-2})=B(\la),\quad\la\in\DD,$$ for some positive numbers $c,d$ satisfying $c^2+d=1$, namely $$c:=\frac{a}{\sqrt{1-a^2}},\quad d:=\frac{b}{1-a^2}.$$ We may assume additionally that $G(0)=0$. Then $B(\la)=\la^{m-3}m_\gamma(\la)$ for some $\gamma\in\DD$ (with exactness up to a unimodular constant; the case $B(\la)=\la^{m-3}$ does not hold). Expanding $G(z_2,z_3)=\alpha z_2+\beta z_3+\ldots$, we get \begin{align*}\alpha c&=-\gamma,\\\beta d&=1-|\gamma|^2.\end{align*} Therefore, $\beta(1-c^2)=1-|\al|^2c^2$ or \begin{equation}\label{38}\beta(1-c^2)+|\al|^2c^2=1.\end{equation}

Proceeding as in the proof of Proposition \ref 1, we show that $\alpha z_2+\beta z_3\in\DD$ for any $z_2,z_3$ with $|z_2|^2+|z_3|<1$. In particular, $|\alpha|,|\beta|\leq 1$. It is obvious that it can not be $|\alpha|=|\beta|=1$, whence \eqref{38} fails.
\end{proof}

By the Lempert theorem, any weak 2-extremal of a convex domain is a 2-geodesic. For all $m$, one-dimensional counterexamples (Proposition \ref{27}) are easy to generalize. Namely, let $D\su\CC$ be a non-simply connected taut domain and let $f:\DD\lon D$ be a~weak $m$-extremal. Take a domain $G\su\CC^n$ and a map $g\in\OO(\DD,G)$ with $g(\DD)\su\su G$. Then $(f,g):\DD\lon D\times G$ is a weak $m$-extremal, but not an $m$-extremal (Lemma \ref{cart}). We are not able to decide whether such a situation is possible for $m\ge 3$ in a~convex domain (P\ref{33}).

We present a non-convex, but topologically contractible counterexample, which follows from the following theorem.

\begin{thm}[\cite{z}, Theorem 4.1.1]\label{41}
A complex ellipsoid $\ee(p)$ is convex if and only if $$\bs\ell_{\ee(p)}(\la_1a,\la_2a)=\bs p(\la_1,\la_2),\quad a\in\pa\ee(p),\ \la_1,\la_2\in\DD.$$
\end{thm}

\begin{cor}\label{nc}
Let $\ee(p)$ be non-convex. Then there exists $a\in\pa\ee(p)$ such that for any Blaschke product $B$ of degree $m-1$, having all zeros different, the mapping $Ba:\DD\lon\ee(p)$ is a weak $m$-extremal, but not an $m$-extremal.
\end{cor}

\begin{proof}
By Lemma \ref{11}, for any $a\in\pa D$ the map $f_a(\la):=\la a$ is a weak 2-extremal for 0 and $\mu\in\DD_*$, so we get weak $m$-extremality of $Ba$ thanks to Lemma \ref{10}$(b)$. On the other side, from Theorem \ref{41} it follows that there exists $a\in\pa\ee(p)$ such that $f_a$ is not a 2-extremal. Therefore, if $Ba$ were an $m$-extremal, making use of Corollary \ref{17} we would get the opposite statement.
\end{proof}

\bigskip

A. Edigarian \cite{e} gave a powerful tool for studying extremal problems of type $(\mathcal P_m)$. First, the author introduced a problem $(\mathcal P)$. Let $D\su\CC^n$ be a bounded domain. A holomorphic mapping $f:\DD\lon D$ is called an \emph{extremal} for $(\mathcal P)$, if $\Phi_j(f)=a_j\in\RR$, $j=1,\ldots,N$, and there is no $h\in\OO(\DD,D)$ such that $\Phi_j(h)=a_j$, $j=1,\ldots,N$, and $h(\DD)\su\su D$; $\Phi_1,\ldots,\Phi_N$ are some functionals. The mappings $g\longmapsto\re g(\la_j)$ and $g\longmapsto\im g(\la_j)$, $j=1,\ldots,m$, for some distinct $\la_1,\ldots,\la_m\in\DD$ ($N=2m$), are model examples of such functionals. In that case it is natural, due to the Cauchy formula, to count how many $\la_j$'s are different from 0. This number is specified by writing $(\mathcal P_{m-1})$ or $(\mathcal P_m)$ (it may be defined for other problems $(\mathcal P)$). We have the following relationship with weak $m$-extremals.

\begin{rem}[cf. \cite{e}, Lemma 20 and \cite{jp}, Remark 11.4.4]
Let $D\su\CC^n$ be a~bounded domain. Then a holomorphic map $f:\DD\lon D$ is a weak $m$-extremal for $m$ non-zero points if and only if it is an extremal for model $(\mathcal P_m)$. Otherwise, if one of $m$ points is 0, we have equivalently an extremal for model $(\mathcal P_{m-1})$.
\end{rem}

A theorem of A. Edigarian delivers a necessary form of extremals $f:\DD\lon\ee(p)$ for $(\mathcal P_{m-1})$ (for convenience we write $m-1$ instead of $m$ and change the formulation). 

\begin{thm}[\cite{e}, Theorem 4]\label{edi}
Let $f:\DD\lon\ee(p)$ be an extremal for $(\mathcal P_{m-1})$ such that $f_j\not\equiv 0$, $j=1,\ldots,n$. Then \begin{equation}\label{22}f_j(\lambda)=a_j\prod_{k=1}^{m-1}\left(\frac{\lambda-\alpha_{kj}}{1-\ov\alpha_{kj}\lambda}\right)^{r_{kj}}
\left(\frac{1-\ov\alpha_{kj}\lambda}{1-\ov\alpha_{k0}\lambda}\right)^{1/p_j},\quad j=1,\ldots,n,\end{equation} where \begin{align*}
&a_j\in\CC_*,\quad\alpha_{kj}\in\CDD,\quad\alpha_{k0}\in\DD,\quad r_{kj}\in\{0,1\},\\
&\sum_{j=1}^n|a_j|^{2p_j}\prod_{k=1}^{m-1}(\la-\alpha_{kj})(1-\ov\alpha_{kj}\la)=\prod_{k=1}^{m-1}(\la-\alpha_{k0})(1-\ov\alpha_{k0}\la),\quad\la\in\CC,\\
&f\neq const.\end{align*}
\end{thm}

\begin{rem}
Let $f$ be of the form \eqref{22}.
\begin{enumerate}[$(a)$]
\item We omit the condition $r_{kj}=1\Longrightarrow\alpha_{kj}\in\DD$ from the paper of A. Edigarian. It has no matter for our consideration, since for $\al_{kj}\in\TT$ the function $m_{\alpha_{kj}}$ extends as a~unimodular constant.
\item Originally, there is no condition $\al_{k0}\in\DD$, but $\al_{k0}\in\CDD$. However, if $\al_{\wi k0}\in\TT$ for some $\wi k$, then from the equality $$\sum_{j=1}^n|a_j|^{2p_j}\prod_{k=1}^{m-1}|\la-\alpha_{kj}|^2=\prod_{k=1}^{m-1}|\la-\alpha_{k0}|^2,\quad\la\in\TT,$$ we deduce that for any $j=1,\ldots,n$ there exists $k_j\in\{1,\ldots,{m-1}\}$ such that $\al_{k_jj}=\al_{\wi k0}$. Then the corresponding factor for $k_j$ and $j$ in \eqref{22} is a unimodular constant. We redefine $\al_{k_jj}$ and $\al_{\wi k0}$ to be the same element of $\DD$ (or remove) and repeat the procedure if needed.
\item The map $f$ extends continuously to $\CDD$. In particular, $f$ is proper.
\end{enumerate} 
\end{rem}

\begin{prop}[\cite{jp}, Proposition 16.2.2, \cite{jpz}]
Let $\ee(p)$ be convex and let $f:\DD\lon\ee(p)$ be a holomorphic mapping. Then, if $f_j\not\equiv 0$, $j=1,\ldots,n$, it follows that $f$ is a $2$-extremal $($i.e. a $2$-geodesic$)$ if and only if it is of the form \eqref{22} with $m=2$.
\end{prop}

It is not known in a general situation whether mappings given by \eqref{22} are even some weak $l$-extremals (P\ref{35}). Our aim is to present solutions of particular cases.

We have new examples of convex domains, in which weak $m$-extremality implies $m$-extremality. Define the set $\sss_n\su[1/2,\infty)^n$ as follows:
\begin{itemize}
\item $(1/2,\ldots,1/2)\in\sss_n$,
\item $p\in\sss_n$, $c\ge 1\Longrightarrow cp\in\sss_n$,
\item $p\in\sss_n$ $\Longrightarrow$ $(p_1,\ldots,p_{j-1},1,p_{j+1},\ldots,p_n)\in\sss_n$ for any $j$.
\end{itemize} 

\begin{lem}
$(p_0,\ldots,p_0)\in\sss_n$ for $p_0\ge 1/2$ and $[1,\infty)^n\su\sss_n$.
\end{lem}

\begin{proof}
The first claim is obvious. For the second one assume that $1\le p_1\le\ldots\le p_n$ and put $c_1=p_1$, $c_j=p_j/p_{j-1}$, $j\ge 2$. Then the following sequences belong to $\sss_n$: 
$(c_n,\ldots,c_n)$,\\
$(1,\ldots,1,c_n)$,\\
$(c_{n-1},\ldots,c_{n-1},c_{n-1}c_n)$,\\
$(1,\ldots,1,c_{n-1},c_{n-1}c_n)$,\\
$(c_{n-2},\ldots,c_{n-2},c_{n-2}c_{n-1},c_{n-2}c_{n-1}c_n)$,\\
$(1,\ldots,1,c_{n-2},c_{n-2}c_{n-1},c_{n-2}c_{n-1}c_n)$,\\
$\vdots$\\
$(1,1,c_3,c_3c_4,\ldots,c_3\ldots c_{n-1},c_3\ldots c_n)$,\\
$(c_2,c_2,c_2c_3,c_2c_3c_4,\ldots,c_2c_3\ldots c_{n-1},c_2c_3\ldots c_n)$,\\
$(1,c_2,c_2c_3,c_2c_3c_4,\ldots,c_2c_3\ldots c_{n-1},c_2c_3\ldots c_n)$,\\
$(c_1,c_1c_2,c_1c_2c_3,\ldots,c_1c_2c_3\ldots c_{n-1},c_1c_2c_3\ldots c_n)=p$.
\end{proof}

\begin{prop}\label{31}
Let $p\in\sss_n$ and let $f:\DD\lon\ee(p)$ be a holomorphic map. Then
\begin{enumerate}[$(a)$]
\item $f$ is a weak $m$-extremal if and only if it is an $m$-extremal.
\item if $f_j\not\equiv 0$, $j=1,\ldots,n$, it follows that $f$ is an $m$-extremal if and only if it is of the form \eqref{22}.
\end{enumerate}

In particular, $(a)$ and $(b)$ hold in any symmetric convex and in any $\cC^2$-smooth complex ellipsoid.
\end{prop}

\begin{proof}
It is sufficient to prove the claim for $p=(1/2,\ldots,1/2)$ and that if the claim holds for $p\in[1/2,\infty)^n$ then it also holds for
\begin{enumerate}[$(i)$]
\item $cp$ with $c>1$, 
\item $q=(1,p_2,\ldots,p_n)$, $n\ge 2$.
\end{enumerate}

Let $p=(1/2,\ldots,1/2)$. One can assume that $f$ is a weak $m$-extremal for 0 and some other $m-1$ points and $f_j\not\equiv 0$ for any $j$. Then $f$ is of the form \eqref{22}. Losing no generality, $a_j>0$. Consider the map $g:\DD\lon\ee(1/2)\su\CC^n$ given by $$g_j(\lambda):=a_j\prod_{k=1}^{m-1}\frac{\lambda-\alpha_{kj}}{1-\ov\alpha_{kj}
\lambda}
\left(\frac{1-\ov\alpha_{kj}\lambda}{1-\ov\alpha_{k0}\lambda}
\right)^2.$$ Then $$g_j(\lambda)=
a_j\prod_{k=1}^{m-1}\frac{\lambda-\alpha_{kj}}{\la-\alpha_{k0}}
\frac{1-\ov\alpha_{kj}\lambda}{1-\ov\alpha_{k0}\lambda}\prod_{k=1}^{m-1}\frac{\lambda-\alpha_{k0}}{1-\ov\alpha_{k0}\lambda}.$$ Putting $F(z):=z_1+\ldots+z_n$, we have $$F(g(\la))=\prod_{k=1}^{m-1}\frac{\lambda-\alpha_{k0}}{1-\ov\alpha_{k0}\lambda},$$ which shows that $g$ is an $m$-geodesic. After iterating Lemma \ref{10}$(a)(i)$, we get $m$-extremality of $f$. We proved in fact that any mapping of the form \eqref{22} is an $m$-extremal, so $(b)$ follows.

$(i)$ Suppose that the claim is true for $p$, but not for $cp$. One can assume that $f:\DD\lon\ee(cp)$ is a weak $m$-extremal for 0 and some other $m-1$ points and $f_j\not\equiv 0$ for any $j$. Then $f$ is of the form \eqref{22} with $cp$ instead of $p$. There exist different points $\la_1,\ldots,\la_m\in\DD$ and a mapping $h\in\OO(\DD,\ee(cp))$ with $h(\la_l)=f(\la_l)$, $l=1,\ldots,m$, and $h(\DD)\su\su\ee(cp)$. Then $$h_j(\la_l)a_j^{c-1}\prod_{k=1}^{m-1}\left(\frac{1-\ov\alpha_{kj}\lambda_l}{1-\ov\alpha_{k0}\lambda_l}
\right)^{1/p_j-1/(cp_j)}=a_j^c\prod_{k=1}^{m-1}\left(\frac{\lambda_l-\alpha_{kj}}{1-\ov\alpha_{kj}
\lambda_l}\right)^{r_{kj}}\left(\frac{1-\ov\alpha_{kj}\lambda_l}{1-\ov\alpha_{k0}\lambda_l}
\right)^{1/p_j}.$$ Note that $g:\DD\lon\CC^n$ defined as $$g_j(\la):=h_j(\la)a_j^{c-1}\prod_{k=1}^{m-1}\left(\frac{1-\ov\alpha_{kj}\lambda}{1-\ov\alpha_{k0}\lambda}
\right)^{1/p_j-1/(cp_j)}$$ satisfies $g(\DD)\su\su\ee(p)$. Indeed, let $d$ satisfies $1/c+1/d=1$, that is $d:=\dfrac{c}{c-1}>0$. By the H\"older inequality and the maximum principle we have \begin{align*}\sum_{j=1}^n|g_j(\la)|^{2p_j}&=\sum_{j=1}^n|h_j(\la)|^{2p_j}|a_j|^{2(c-1)p_j}\prod_{k=1}^{m-1}\left|\frac{1-\ov\alpha_{kj}\lambda}{1-\ov\alpha_{k0}\lambda}
\right|^{2-2/c}\\
&\le\left(\sum_{j=1}^n|h_j(\la)|^{2cp_j}\right)^{1/c}\left(\sum_{j=1}^n\left(|a_j|^{2(c-1)p_j}
\prod_{k=1}^{m-1}\left|\frac{1-\ov\alpha_{kj}\lambda}{1-\ov\alpha_{k0}\lambda}
\right|^{2(c-1)/c}\right)^{d}\right)^{1/d}\\
&\le C^{1/c}\left(\sum_{j=1}^n|a_j|^{2cp_j}
\prod_{k=1}^{m-1}\left|\frac{1-\ov\alpha_{kj}\lambda}{1-\ov\alpha_{k0}\lambda}
\right|^2\right)^{1/d}\\&\le C^{1/c}<1,\end{align*}
where $C:=\sup_\DD\sum_{j=1}^n|h_j|^{2cp_j}<1$. It follows that $\wi f:\DD\lon\ee(p)$ given as $$\wi f_j(\la):=a_j^c\prod_{k=1}^{m-1}\left(\frac{\lambda-\alpha_{kj}}{1-\ov\alpha_{kj}
\lambda}\right)^{r_{kj}}\left(\frac{1-\ov\alpha_{kj}\lambda}{1-\ov\alpha_{k0}\lambda}
\right)^{1/p_j}$$ is not a weak $m$-extremal for $\la_1,\ldots,\la_m$, contradiction.

$(ii)$ It suffices to show that any map $f:\DD\lon\ee(q)$ of the form \eqref{22} is an $m$-extremal. Due to Lemma \ref{10}$(a)(i)$ one may assume that $r_{kj}=1$ for any $k,j$.

Note that $\DD\not\su f_j(\CDD)$, $j=1,\ldots,n$. Otherwise, for some $j$ and any $\zeta\in\TT$ we would find a sequence $\mu_l\in\CDD$ such that $f_j(\mu_l)=(1-1/l)\zeta$. Passing to a subsequence we can assume that $\mu_l\to\mu\in\CDD$. Then $f_j(\mu)=\zeta$, so $\mu\in\TT$ and $f_{j'}(\mu)=0$ for $j'\neq j$. Since different $\zeta$'s give different $\mu$'s, this implies that $f_{j'}$ has infinitely many zeros on $\TT$, contradiction. 

Let $\mu\in\DD\sm f_1(\CDD)$ and consider the following automorphism of $\ee(q)$: $$A(z):=\left(\frac{z_1-\mu}{1-\ov\mu z_1},z_2\left(\frac{\sqrt{1-|\mu|^2}}{1-\ov\mu z_1}\right)^{1/p_2},\ldots,z_n\left(\frac{\sqrt{1-|\mu|^2}}{1-\ov\mu z_1}\right)^{1/p_n}\right).$$ We claim that $\wi f:=A\circ f$ is of the form \eqref{22}. Note that $$f_1(\la)=a_1\prod_{k=1}^{m-1}\frac{\la-\al_{k1}}{1-\ov\al_{k0}\la},\quad\quad\wi f_1(\la)=\frac{f_1(\la)-\mu}{1-\ov\mu f_1(\la)}$$ and the polynomials of variable $\la$ $$a_1\prod_{k=1}^{m-1}(\la-\al_{k1})-\mu\prod_{k=1}^{m-1}(1-\ov\al_{k0}\la),\quad\prod_{k=1}^{m-1}(1-\ov\al_{k0}\la)-\ov\mu a_1\prod_{k=1}^{m-1}(\la-\al_{k1})$$
do not vanish in $\ov\DD$. Therefore $$f_1(\la)-\mu=c_1\prod_{k=1}^{m-1}\frac{1-\ov\beta_{k1}\la}{1-\ov\al_{k0}\la},\quad c_1\in\CC_*,\,\beta_{k1}\in\DD,$$ and $$1-\ov\mu f_1(\la)=c_2\prod_{k=1}^{m-1}\frac{1-\ov\beta_{k0}\la}{1-\ov\al_{k0}\la},\quad c_2\in\CC_*,\,\beta_{k0}\in\DD.$$ This leads to $$\wi f_1(\la)=b_1\prod_{k=1}^{m-1}\frac{1-\ov\beta_{k1}\la}{1-\ov\beta_{k0}\la},\quad b_1\in\CC_*,\,s_{k1}=0,$$ and \begin{align*}\wi f_j(\la)&=a_j\prod_{k=1}^{m-1}\frac{\lambda-\alpha_{kj}}{1-\ov\alpha_{kj}\lambda}
\left(\frac{1-\ov\alpha_{kj}\lambda}{1-\ov\alpha_{k0}\lambda}\right)^{1/p_j}
\left(\frac{\sqrt{1-|\mu|^2}}{c_2}\prod_{k=1}^{m-1}\frac{1-\ov\al_{k0}\la}{1-\ov\beta_{k0}\la}\right)^{1/p_j}\\
&=b_j\prod_{k=1}^{m-1}\frac{\lambda-\beta_{kj}}{1-\ov\beta_{kj}\lambda}\left(\frac{1-\ov\beta_{kj}\lambda}{1-\ov\beta_{k0}\lambda}\right)^{1/p_j}, \quad b_j\in\CC_*,\,\beta_{kj}=\al_{kj},\,s_{kj}=1,\,j\ge 2.\end{align*} The last two conditions of \eqref{22} follow trivially from the fact that $\wi f(\TT)\su\pa\ee(q)$.

Suppose that there are different points $\la_1,\ldots,\la_m\in\DD$ and a map $h\in\OO(\DD,\ee(q))$ with $h(\la_l)=\wi f(\la_l)$ for any $l$ and $h(\DD)\su\su\ee(q)$. For $t\in\CC$ the map $\wi h:=th+(1-t)\wi f$ satisfies $\wi h(\la_l)=\wi f(\la_l)$, $l=1,\ldots,m$. However, for small $t\in(0,1)$ the function $\wi h_1$~does not vanish in $\DD$. Define a holomorphic map $g:\DD\lon\CC^n$ by $g_j:=\wi h_j^{q_j/p_j}$, $j=1,\ldots,n$. The Jensen inequality implies that \begin{multline*}\sum_{j=1}^n|g_j|^{2p_j}=\sum_{j=1}^n|th_j+(1-t)\wi f_j|^{2q_j}\\\leq t\sum_{j=1}^n|h_j|^{2q_j}+(1-t)\sum_{j=1}^n|\wi f_j|^{2q_j}\le tC+1-t<1,\end{multline*} where $C:=\sup_\DD\sum_{j=1}^n|h_j|^{2q_j}<1$. Thus $g(\DD)\su\su\ee(p)$. Further we have $$g_j(\la_l)=\zeta_jb_j^{q_j/p_j}\prod_{k=1}^{m-1}\left(\frac{\lambda_l-\beta_{kj}}{1-\ov\beta_{kj}
\lambda_l}\right)^{(q_j/p_j)s_{kj}}\left(\frac{1-\ov\beta_{kj}\lambda_l}{1-\ov\beta_{k0}\lambda_l}\right)^{1/p_j}$$ for some $\zeta_j\in\TT$. This contradicts the fact that $\wi g:\DD\lon\ee(p)$, $$\wi g_j(\la):=b_j^{q_j/p_j}\prod_{k=1}^{m-1}\left(\frac{\lambda-\beta_{kj}}{1-\ov\beta_{kj}
\lambda}\right)^{(q_j/p_j)s_{kj}}
\left(\frac{1-\ov\beta_{kj}\lambda}{1-\ov\beta_{k0}\lambda}
\right)^{1/p_j},$$ is an $m$-extremal. Therefore, the map $\wi f$ is an $m$-extremal, so $f$ is, as well.
\end{proof}

\begin{rem}
Note that $(1/2,p_2)\notin\sss_2$ for $p_2>1/2$, $p_2\neq 1$. However, the convex case $n=2$ in Proposition \ref{31} would follow from the claim for such pairs.
\end{rem}

We suppose that weak $m$-extremality coincides with $m$-extremality in any convex complex ellipsoid (P\ref{ep}).

\begin{prop}\label{30}
Let $f:\DD\lon\ee(p)$ be given by \eqref{22}. Assume that 
\begin{enumerate}[$(a)$]
\item $p_1,\ldots,p_n\ge 1/2$,
\item $q\in\sss_n$,
\item $\alpha_{kj}\in\DD$, $r_{kj}=0$ for $k=1,\ldots,m-1$ and $j\in J$, where $J:=\{j:p_j/q_j\notin\NN\}$,
\item $s_j:=\#\{k:r_{kj}=1\}$,
\item $\wi m:=m+\sum_{j\notin J}(p_j/q_j-1)s_j$.
\end{enumerate}
Then $f$ is an $\wi m$-extremal.

In particular, $f$ is an $(m+(m-1)(p_1/q_1+\ldots+p_n/q_n-n))$-extremal, provided that $q\in\sss_n$ and $p_j/q_j\in\NN$, $j=1,\ldots,n$.
\end{prop}

\begin{proof}
Suppose contrary. Note that $f_j\neq 0$ in $\CDD$, $j\in J$. Proceeding like in the proof of Proposition \ref{31}$(a)$ (now functions $\wi h_j^{p_j/q_j}$ are well-defined), we obtain that $\wi g:\DD\lon\ee(q)$ defined by $$\wi g_j(\la):=a_j^{p_j/q_j}\prod_{k=1}^{m-1}\left(\frac{\lambda-\alpha_{kj}}{1-\ov\alpha_{kj}\lambda}\right)^{(p_j/q_j)r_{kj}}
\left(\frac{1-\ov\alpha_{kj}\lambda}{1-\ov\alpha_{k0}\lambda}\right)^{1/q_j}$$ is not an $\wi m$-extremal.

On the other side, by Proposition \ref{31}$(b)$, the mapping $g:\DD\lon\ee(q)$, $$g_j(\la):=a_j^{p_j/q_j}\prod_{k=1}^{m-1}\left(\frac{\lambda-\alpha_{kj}}{1-\ov\alpha_{kj}
\lambda}\right)^{r_{kj}}\left(\frac{1-\ov\alpha_{kj}\lambda}{1-\ov\alpha_{k0}\lambda}
\right)^{1/q_j},$$ is an $m$-extremal. We use $\sum_{j\notin J}(p_j/q_j-1)s_j$ times Lemma \ref{10}$(b)$ and Proposition \ref{31}$(a)$ to get $\wi m$-extremality of $\wi g$, contradiction.
\end{proof}

\begin{rem}
Note the fact following from the proof of Proposition \ref{30}. Suppose that $p,q\in\RR_{>0}^n$ are such that $p_j/q_{\sigma(j)}\in\NN$, $j=1,\ldots,n$, for some permutation $\sigma$ of $\{1,\ldots,n\}$ (it is equivalent to the existence of a proper holomorphic map between $\ee(p)$ and $\ee(q)$). Assume that any map of the form \eqref{22} in $\ee(q)$ is some (weak) $t$-extremal. Then any map given by \eqref{22} in $\ee(p)$ is some (weak) $s$-extremal. However, this procedure delivers the same $p$ as described in Proposition \ref{30}.
\end{rem}

\begin{prop}\label{36}
Let $f:\DD\lon\ee(p)$ be of the form \eqref{22}. Assume that 
\begin{enumerate}[$(a)$]
\item $p_1,\ldots,p_n\ge 1/2$,
\item $q\in\sss_n$,
\item $\alpha_{kj}\in\DD$ for $k=1,\ldots,m-1$ and $j\in J$, where $J:=\{j:p_j/q_j\notin\NN\}$,
\item $S:=\{(k,j):r_{kj}=1\}$,
\item $\alpha_{kj}$, $(k,j)\in S$, are distinct,
\item $s:=\#S\ge m$.
\end{enumerate}
Then $f$ is a weak $s$-extremal for $\alpha_{kj}$, $(k,j)\in S$.
\end{prop}

\begin{proof}
Suppose that there exists a holomorphic mapping $h:\DD\lon\ee(p)$ with $h(\alpha_{kj})=f(\alpha_{kj})$, $(k,j)\in S$, and $h(\DD)\su\su\ee(p)$. In particular, $h_j(\alpha_{kj})=0$, $(k,j)\in S$. Consider the maps $g:\DD\lon\ee(p)$ and $\wi f:\DD\lon\ov{\ee(p)}$ given as $$g_j(\la):=\frac{h_j(\la)}{\prod_{k=1}^{m-1}\left(\frac{\lambda-\alpha_{kj}}{1-\ov\alpha_{kj}
\lambda}\right)^{r_{kj}}},$$$$\wi f_j(\la):=a_j\prod_{k=1}^{m-1}\left(\frac{1-\ov\alpha_{kj}\lambda}{1-\ov\alpha_{k0}\lambda}
\right)^{1/p_j}.$$ We have $g(\alpha_{kj})=\wi f(\alpha_{kj})$, $(k,j)\in S$, and $g(\DD)\su\su\ee(p)$. It follows that $\wi f(\DD)\su\ee(p)$, as otherwise $\wi f$ would be a constant lying in the boundary of $\ee(p)$ (it would also contradict the condition $s\ge m$). Hence $\wi f$ is not a weak $s$-extremal for $\alpha_{kj}$, $(k,j)\in S$. However, by Proposition \ref{30}, the mapping $\wi f$ is an $m$-extremal. This is impossible, since $s\ge m$. 
\end{proof}

\bigskip

In the sequel (see also Proposition \ref{24}) occur non-constant mappings of the form $(a_1B_1,\ldots,a_nB_n)$, where $a\in\pa\ee(p)$ and $B_1,\ldots,B_n$ are finite Blaschke products. We think that any $m$-extremal of the ball is equivalent with some of these maps (P\ref{25}), which are suspected to be some $k$-geodesics (P\ref{26}). This would give a positive answer for (P\ref{mkball}).

\begin{rem}
Some $m$-extremality of maps $(a_1B_1,\ldots,a_nB_n)$ in convex complex ellipsoids follows from 2-geodesity of the mapping $\la\longmapsto\la a$ and Lemma \ref{10}$(a)(i)$.
\end{rem}

\begin{prop}\label{bl}
Let $a\in\pa\ee(p)$ be such that $$(p_j|a_j|^{2p_j})_{j=1}^n=c(m_1,\ldots,m_n),\quad c>0,\ m_j\in\NN.$$ Assume that $B_1,\ldots,B_n$ are finite Blaschke products, not all constant. Then the map $(a_1B_1,\ldots,a_nB_n):\DD\lon\ee(p)$ is some $m$-geodesic.
\end{prop}

\begin{proof}
Consider the logarithmic image of $\ee(p)$, that is the convex domain \begin{align*}\Omega:&=\{x\in\RR^n:(e^{x_1},\ldots,e^{x_n})\in\ee(p)\}\\
&=\left\{x\in\RR^n:\sum_{j=1}^ne^{2p_jx_j}<1\right\}.\end{align*} The affine tangent space at $b:=(\log|a_1|,\ldots,\log|a_n|)\in\pa\Omega$ is $$\left\{x\in\RR^n:\sum_{j=1}^np_je^{2p_jb_j}(x_j-b_j)=0\right\}=\left\{x\in\RR^n:\sum_{j=1}^ncm_j(x_j-b_j)=0\right\},$$ whence $$\Omega\su\left\{x\in\RR^n:\sum_{j=1}^nm_j(x_j-b_j)<0\right\}.$$ This implies
\begin{align*}\ee(p)&\su\left\{z\in\CC^n:\sum_{j=1}^nm_j\log|z_j|<\sum_{j=1}^nm_jb_j\right\}\\
&=\left\{z\in\CC^n:\prod_{j=1}^n|z_j|^{m_j}<\prod_{j=1}^n|a_j|^{m_j}\right\},\end{align*} so the polynomial $$F(z):=\prod_{j=1}^n\left(\frac{z_j}{a_j}\right)^{m_j}$$ is an $m$-left inverse we are looking for.
\end{proof}

\begin{prop}\label{bl1}
Let $a\in\pa\ee(p)$ and let $m$ be the least common multiplicity of numbers $m_1,\ldots,m_n\in\NN$. Assume that $2p_jm_j\geq m$, $j=1,\ldots,n$. Then the map $\DD\ni\la\longmapsto(a_1\la^{m_1},\ldots,a_n\la^{m_n})\in\ee(p)$ is an $(m+1)$-geodesic.
\end{prop}

\begin{proof}
One may assume that $a_j\in(0,1)$. Define the domain \begin{align*}\Omega:&=\{x\in\RR^n:(x_1^{m_1/m},\ldots,x_n^{m_n/m})\in\ee(p)\}\\
&=\left\{x\in\RR^n:\sum_{j=1}^nx_j^{\frac{2p_jm_j}{m}}<1\right\},\end{align*} which is convex. The affine tangent space at $b:=(a_1^{m/m_1},\ldots,a_n^{m/m_n})\in\pa\Omega$~is $$\left\{x\in\RR^n:\sum_{j=1}^np_jm_jb_j^{\frac{2p_jm_j}{m}-1}(x_j-b_j)=0\right\},$$ whence $$\Omega\su\left\{x\in\RR^n:\sum_{j=1}^np_jm_jb_j^{\frac{2p_jm_j}{m}-1}(x_j-b_j)<0\right\}.$$ It follows that
$$\ee(p)\su\left\{z\in\CC^n:\sum_{j=1}^np_jm_jb_j^{\frac{2p_jm_j}{m}-1}|z_j|^{\frac{m}{m_j}}<\sum_{j=1}^np_jm_jb_j^{\frac{2p_jm_j}{m}}\right\},$$ so the polynomial $$F(z):=\frac{\sum_{j=1}^np_jm_jb_j^{\frac{2p_jm_j}{m}-1}z_j^{\frac{m}{m_j}}}{\sum_{j=1}^np_jm_jb_j^{\frac{2p_jm_j}{m}}}$$ is an $(m+1)$-left inverse.
\end{proof}

\section{The Euclidean ball}\label{23}
We say that holomorphic mappings $f,g:\DD\lon\BB_n$ are \emph{equivalent} if there exists $A\in\aut(\BB_n)$ such that $f=A\circ g$.

Recall that the automorphism group of the ball consists of the mappings $U\circ\chi_w$ (equivalently, of the mappings $\chi_w\circ U$), where $U:\CC^n\lon\CC^n$ is unitary and $\chi_w:\BB_n\lon\BB_n$ defined as $\chi_0:=\id_{\BB_n}$ and $$\label{chi}\chi_w(z):=\frac{1}{|w|^2}\frac{\sqrt{1-|w|^2}(|w|^2z-\lan z,w\ran w)-|w|^2w+\lan z,w\ran w}{1-\lan z,w\ran},\quad w\in{\BB_n}_*.$$

\begin{rem}
Any 2-extremal $f:\DD\lon\BB_n$ is equivalent to $\la\longmapsto(\la,0,\ldots,0)$.
\end{rem}

\begin{rem}[\cite{kz}]
\begin{enumerate}[$(a)$]
\item Any 3-extremal $f:\DD\lon\BB_n$, $n\ge 2$, is equivalent with some map \begin{equation}\label{eq}g:\DD\ni\la\longmapsto(a\la,\sqrt{1-a^2}\la m_\alpha(\la),0,\ldots,0)\in\BB_n,\end{equation} where $0\le a\le 1$ and $\al\in\DD$ (take $A\in\aut(\BB_n)$ such that $A(f(0))=0$, divide by $\la$ to get either a 2-extremal or a constant from the boundary, unitarily transform in such a way that some two points of this 2-extremal have the same first coordinate and use the form of 2-extremals). 
\item Any map of the form \eqref{eq} is a 3-extremal. 
\item A mapping given by \eqref{eq} is a 2-extremal if and only if $a=1$.
\item Any 3-extremal is equivalent with exactly one map of the form \eqref{eq}.
\item For $\al=0$ the map given by \eqref{eq} is a 3-geodesic, since it has the 3-left inverse $$F(z):=\frac{1}{2-a^2}z_1^2+\frac{2\sqrt{1-a^2}}{2-a^2}z_2.$$
\end{enumerate}
\end{rem}

By the Schur's algorithm we have the following characterization.

\begin{rem}[\cite{kz}]
Let $f:\DD\lon\BB_n$ be a holomorphic mapping. Then $f$ is an $m$-extremal if and only if $$f(\la)=A_1(\la A_2(\la\ldots A_l(\la a)\ldots)),\quad\la\in\DD,$$ for some $A_1,\ldots,A_l\in\aut(\BB_n)$, $1\le l\le m-1$ and $a\in\pa\BB_n$. In particular, any $m$-extremal of $\BB_n$ extends holomorphically to a neighborhood of $\CDD$.
\end{rem}

\begin{rem}[cf. the proof of Proposition \ref{31}]
Any $m$-extremal of $\BB_n$, $n\ge 2$, is equivalent with a map, whose coordinates have no zeros in a neighborhood of $\CDD$.
\end{rem}

Recall less obvious facts.

\begin{prop}[\cite{kz}, Proposition 8]
Any weak $m$-extremal of\, $\BB_n$ is an $m$-extremal.
\end{prop}

\begin{prop}[\cite{kz}, Proposition 11]\label{mm}
Let $m\ge 4$ and $0<a<1$. Then the mapping $$f(\la):=(a\la^{m-2},\sqrt{1-a^2}\la^{m-1}),\quad\la\in\DD,$$ is an $m$-extremal, but not an $m$-geodesic of\, $\BB_2$.
\end{prop} 

\begin{rem}
The fundamental Poincar\'e theorem states that $\BB_n$ and $\DD^n$ are not biholomorphic if $n\ge 2$. Note that a new proof of this fact follows from Remark \ref{pol} and Proposition \ref{mm}.
\end{rem}

The main result of the section is

\begin{thm}\label{32}
Any $3$-extremal of\, $\BB_n$ is a $3$-geodesic.
\end{thm}

\begin{proof}
It suffices to prove the claim for $n=2$. Consider 3-geodesics of the form $$f(\la):=(am_c(\la),bm_c(\la)^2),\quad\la\in\DD,$$ where $a,b\in(0,1)$, $a^2+b^2=1$ and $c\in\DD_*$. Any such mapping is equivalent to $g(\la):=(\alpha\la,\beta\la m_\gamma(\la))$ for some $\alpha,\beta\in[0,1]$, $\alpha^2+\beta^2=1$ and $\gamma\in\DD$, i.e. there are a unitary map $U$ and a point $w\in\BB_2$ such that \begin{equation}\label{bc}\chi_w(am_c(\la),bm_c(\la)^2)=U(\alpha\la,\beta\la m_\gamma(\la)).\end{equation} We will find formulas for $\beta$ and $\gamma$ depending of $b$ and $c$. Then we shall prove that $(\beta,\gamma)$ runs over the whole set $(0,1)\times\DD_*$ as $(b,c)$ runs over it. This will let us `invert' $g$, since we are able to do it with $f$.

Taking $\la:=0$ in \eqref{bc} we get $w=(-ac,bc^2)$. Note that $\beta\neq 0$, since otherwise $\la:=c$ gives $\chi_w(0)=U(c,0)$; hence $|w|^2=|c|^2$, i.e. $a^2+b^2|c|^2=1$, contradiction.

By the formula for $\chi_w$ we have \begin{multline}\label{7}p_0+p_1m_c(\la)+p_2m_c(\la)^2\\=(1+a^2\ov cm_c(\la)-b^2\ov c^2m_c(\la)^2)\la(q_1\alpha+q_2\beta m_\gamma(\la),q_3\alpha+q_4\beta m_\gamma(\la))\end{multline} for some $p_j\in\CC^2$, $q_j\in\CC$ with $q_2\neq 0$ or $q_4\neq 0$. Therefore, \begin{equation}\label{8} 1+a^2\ov cm_c(1/\ov\gamma)-b^2\ov c^2m_c(1/\ov\gamma)^2=0,\end{equation} unless $\gamma=0$ or $\gamma=c$.

Suppose that $\gamma=0$ and $q_2\neq 0$. Then \begin{align*}p_{01}+p_{11}\la+p_{21}\la^2=(1+(1-b^2)\ov c\la-b^2\ov c^2\la^2)m_{-c}(\la)&(q_1\alpha+q_2\beta m_{-c}(\la))\\=(1-b^2\ov c\la)(c+\la)&(q_1\alpha+q_2\beta m_{-c}(\la)).\end{align*} Since the numbers $$\frac{1}{b^2\ov c},\quad-c,\quad-\frac{1}{\ov c}$$ are different, we infer that the right side has a singularity, contradiction.

The case $\gamma=c$ is also impossible, as otherwise the rank of the singularity $1/\ov c$ on the right side of \eqref 7 would equal 3.

The equation \eqref{8} is equivalent to $$(1-b^2\ov cm_c(1/\ov\gamma))(1+\ov cm_c(1/\ov\gamma))=0,$$ that is $m_c(1/\ov\gamma)=1/(b^2\ov c)$, i.e. $$\gamma=c\frac{1+b^2}{1+b^2|c|^2}=m_{-c}(b^2c).$$

Moreover, there exist a unitary map $\wi U$ and a point $\wi w\in\BB_2$ satisfying $$\wi U(am_c(\la),bm_c(\la)^2)=\chi_{\wi w}(\alpha\la,\beta\la m_\gamma(\la)),\quad\la\in\DD,$$ whence $0=\chi_{\wi w}(\alpha c,\beta cm_\gamma(c))$ and $\wi U(-ac,bc^2)=\chi_{\wi w}(0)$. This implies $$a^2|c|^2+b^2|c|^4=\alpha^2|c|^2+\beta^2|c|^2|m_\gamma(c)|^2,$$ equivalently (using $|m_\gamma(c)|=|m_c(\gamma)|=b^2|c|$) $$1-b^2+b^2|c|^2=\alpha^2+(1-\alpha^2)b^4|c|^2.$$ Therefore, $$\alpha^2=\frac{(1-b^2)(1+b^2|c|^2)}{1-b^4|c|^2},\quad\beta^2=\frac{b^2-b^2|c|^2}{1-b^4|c|^2}=-m_{b^2}(b^2|c|^2).$$

To finish the proof, it suffices to show that the mapping $$h:(0,1)\times\DD_*\ni(b,c)\longmapsto(-m_{b^2}(b^2|c|^2),m_{-c}(b^2c))\in(0,1)\times\DD_*$$ is surjective. It is equivalent to the surjectivity of $$(0,1)^2\ni(b,c)\longmapsto h(b,c)\in(0,1)^2.$$ 

Fix $(p,q)\in(0,1)^2$. Putting $$F(\la):=m_q(\la)-\la m_p(\la m_q(\la)),\quad\la\in\DD,$$ we see that $F(-1,1)\su\RR$, $F(0)=-q<0$ and $F(q)=pq>0$. Thus there exists $c\in(0,q)$ such that $F(c)=0$. Note that $-c<m_q(c)<0$. Let $b\in(0,1)$ satisfy $-b^2=m_q(c)/c$. Then $m_c(q)=b^2c$, i.e. $q=m_{-c}(b^2c)$. Moreover, $$m_{-b^2}(-p)=-m_{-p}(-b^2)=-m_{-p}\left(\frac{m_q(c)}{c}\right)=-cm_q(c)=b^2c^2,$$ so $p=-m_{b^2}(b^2c^2)$.
\end{proof}

In Propositions \ref{bl} and \ref{bl1} some $m$-geodesity of mappings $(a_1B_1,\ldots,a_nB_n)$ was investigated. We add one more positive result.

\begin{prop}\label{24}
Let $m\geq 3$, $0<b\le\frac{1}{m-1}$ and $a:=\sqrt{1-b^2}$. Then the mapping $f(\la):=(a\la,b\la^m)$, $\la\in\DD$, is an $(m+1)$-geodesic of\, $\BB_2$.
\end{prop}

\begin{proof}
Consider the more general situation $f(\la)=(a\la^k,b\la^m)$, $k\ge 1$, $m\ge 3$, and use the Lagrange multipliers to the functions of real variables $F(x,y):=cx^m+dy^k$ and $G(x,y):=x^2+y^2-1$ ($c,d>0$ specified later). We wish $F$ had a global (weak) maximum equal to 1 on the set $\{G=0\}$ at the point $(a,b)$. Denote $H:=F-tG$, where $t>0$ is fixed. From the necessary condition for a local extremum we have \begin{align*}0&=\frac{\pa H}{\pa x}(x,y)=mcx^{m-1}-2tx,\\0&=\frac{\pa H}{\pa y}(x,y)=kdy^{k-1}-2ty,\\1&=x^2+y^2.\end{align*} Excluding for a moment the cases $(1,0)$ and $(0,1)$, we find that (remembering that $1=ca^m+db^k$) $$c=\frac{k}{(ka^2+mb^2)a^{m-2}},\quad d=\frac{m}{(ka^2+mb^2)b^{k-2}}$$ (formally, we define $c,d$ by these formulas). The tangent space at $(a,b)$ is $\RR(b,-a)$, so $(a,b)$ is a local maximum if \begin{align}\label{h}0&>\frac{\pa^2H}{\pa x^2}(a,b)b^2+\frac{\pa^2H}{\pa y^2}(a,b)a^2\nonumber\\&=(m(m-1)ca^{m-2}-2t)b^2+(k(k-1)db^{k-2}-2t)a^2\nonumber\\&=2t(m-2)b^2+2t(k-2)a^2.\end{align} Since $t>0$, we see why only $k=1$ may work; in what follows we assume that $k=1$. In that situation \eqref{h} is equivalent to $b^2<\frac{1}{m-1}$, which is true. 

It remains to check that $F(x,y)\leq 1$ for any $x,y$ satisfying the necessary condition. First, we will show that $F(1,0)$, $F(0,1)\leq 1$, that is $c,d\leq 1$. It occurs that $d\leq 1$ is equivalent to $b\le\frac{1}{m-1}$. For the condition $c\leq 1$ we need that $$1\le(a^2+m(1-a^2))a^{m-2}=ma^{m-2}-(m-1)a^m,$$ so consider the function $g(s):=ms^{m-2}-(m-1)s^m$. It decreases on the interval $\left[\sqrt{1-\frac{1}{m-1}},1\right]\ni a$, so $g(a)>g(1)=1$.

Now let $x,y\neq 0$ satisfy the necessary condition. Then $mcx^{m-2}=2t=d/y$, that is $$yx^{m-2}=\frac{d}{mc}=ba^{m-2}.$$ Define $h(s):=s\sqrt{1-s^2}^{m-2}$. Then $h(y)=h(b)$ and $h$ increases on the interval $\left[0,\sqrt{\frac{1}{m-1}}\right]\ni b$, so $y\geq b$. Our aim is to show that $cx^m+dy\leq 1$, that is \begin{align*}\frac{x^m}{a^{m-2}}+mby&\leq a^2+mb^2,\\\frac{x^mb}{yx^{m-2}}+mby&\leq 1+(m-1)b^2,\\b(1-y^2)+mby^2&\leq y+(m-1)b^2y,
\\0&\leq((m-1)by-1)(b-y).\end{align*} The last inequality holds, since $(m-1)by-1\leq y-1<0$.
\end{proof}

The case $\dfrac{1}{m-1}<b<1$ remains unsolved (P\ref{b}).

\section{Boundary properties}
In this section we discuss (almost) properness of weak $m$-extremals. Thanks to almost properness we conclude their uniqueness in bounded strictly convex domains.

Let $D\su\CC^n$ be a bounded domain and $f:\DD\lon D$ a holomorphic mapping. We say that $f$ is \emph{almost proper} if $f^*(\zeta)\in\pa D$ for almost all $\zeta\in\TT$ with respect to the Lebesgue measure on $\TT$. As usual, $f^*(\zeta):=\lim_{r\to 1^-}f(r\zeta)$ is the non-tangential boundary value of $f$ at $\zeta$, which exists for almost all $\zeta\in\TT$, see \cite{ko}.

A domain $D\su\CC^n$ is called \emph{weakly Runge} if it is bounded and there exists a~domain $G\supset\ov D$ such that for any bounded holomorphic map $f:\DD\lon G$ with $f^*(\TT)\su\su D$ we have $f(\DD)\su\su D$.

\begin{rem}[\cite{ek}, Remark 2]
\begin{enumerate}[$(a)$]
\item A bounded Runge domain is weakly Runge.
\item Let $G\su\CC^n$ be a domain and let $u$ be a plurisubharmonic function in $G$. Assume that $$D:=\{z\in G:u(z)<0\}\su\su G.$$ Then any component of $D$ is a weakly Runge domain.
\end{enumerate}
\end{rem}

\begin{prop}[cf. \cite{ek}, Theorem 1]  
Let $D\su\CC^n$ be a weakly Runge domain and let $f:\DD\lon D$ be a weak $m$-extremal such that for some $\gamma>0$ we have $$\dist(f(\la),\pa D)\ge\gamma(1-|\la|),\quad\la\in\DD.$$ Then for any $\alpha>0$ and $\beta<1$ the set $$Q(\al,\beta):=\{\zeta\in\TT:\dist(f(t\zeta),\pa D)\ge\alpha(1-t)^\beta\text{ for any }t\in(0,1)\}$$ has Lebesgue measure zero on $\TT$. In particular, $f$ is almost proper.
\end{prop}

\begin{proof}
This is a slight modification of the proof of \cite[Theorem 1]{ek}. For the Reader's convenience, we present the whole proof (the first and the last part are mostly copied).

Note that for $\beta_1<\beta_2$ we have $Q(\al,\beta_1)\su Q(\al,\beta_2)$. Without loss of generality one may assume that for some $\al>0$ and $\beta\in(0,1)$ the set $P:=Q(\al,\beta)$ has positive measure. We can assume that $$0<\frac{1}{2\pi}\int_{\{\theta\in(0,2\pi):e^{i\theta}\in P\}}d\theta<1$$ (otherwise we take as $P$ any subset of $Q(\al,\beta)$ of positive measure). We put $$\phi(\la):=\frac{1}{2\pi}\int_{\{\theta\in(0,2\pi):e^{i\theta}\in P\}}\frac{e^{i\theta}+\la}{e^{i\theta}-\la}\,d\theta,\quad\la\in\DD,$$ and check that $\re\phi(\la)>0$ and $\re(1-\phi(\la))>0$. In particular, $\phi^*$ exists almost everywhere \cite[Chapter III, Section C]{ko}. 

Losing no generality assume that $f$ is a weak $m$-extremal for $\la_1,\ldots,\la_{m-1},0$. For $t\in(0,1)$ define $$h_t(\la):=f(t\la)+\sum_{j=1}^{m-1}\left(e^{\gamma_t(\phi(\la)-\phi(\la_j))}\frac{\la}{\la_j}\prod_{k\neq j}\frac{\la-\la_k}{\la_j-\la_k}\right)(f(\la_j)-f(t\la_j)),\quad\la\in\DD,$$ with $\gamma_t\in\RR$ specified later. Then $h_t(\la_l)=f(\la_l)$ for any $l$ and $h_t(0)=f(0)$. Our aim is to show that for all $t\in(0,1)$ sufficiently close to 1 there exists $\gamma_t$ such that $h_t(\DD)\su\su D$. First, we shall prove that $h_t^*(\TT)\su\su D$. 

It is sufficient to have for $t$ close to 1 $$\sum_{j=1}^{m-1}e^{\gamma_t(\re\phi^*(\zeta)-\re\phi(\la_j))}c_j\left|\frac{f(\la_j)-f(t\la_j)}{\la_j}\right|\leq
\begin{cases}\frac{\alpha}{2}(1-t)^\beta,\ &\zeta\in P,\\\frac{\gamma}{2}(1-t),\ &\zeta\in\TT\sm P.\end{cases}$$ Since $c_j|f(\la_j)-f(t\la_j)|\leq\rho|\la_j|(1-t)$, it suffices to have \begin{equation}\label{18}\sum_{j=1}^{m-1}e^{\gamma_t(1-\re\phi(\la_j))}\rho\leq
\frac{\alpha}{2}(1-t)^{\beta-1}\end{equation} and \begin{equation}\label{19}\sum_{j=1}^{m-1}e^{-\gamma_t\re\phi(\la_j)}\rho\leq
\frac{\gamma}{2}.\end{equation} Take $\gamma_t$ such that equality in \eqref{18} holds. Then for $t$ sufficiently close to 1 we also have inequality \eqref{19}. Moreover, $$\|h_t-f(t\cdot)\|_\DD\to 0,\quad t\to 1.$$ Since $D$ is weakly Runge, $h_t(\DD)\su\su D$ for $t$ close enough to 1.

To finish the proof suppose that there exists a set $P\su\TT$ of positive measure
such that for all $\zeta\in P$ we have $\dist(f^*(\zeta),\pa D)>\eps>0$. Put $$P_k:=\{\zeta\in\TT:\dist(f(t\zeta),\pa D)>\eps\text{ for any }t\in(1-1/k,1)\},\quad k\in\NN.$$ Then $P\su\bigcup_{k=1}^\infty P_k$. Hence, for some $k$ the set $P_k$ is of positive measure, contradiction.
\end{proof}

\begin{cor} 
Any weak $m$-extremal of a bounded convex domain $D\su\CC^n$ is almost proper.
\end{cor}

\begin{proof}
Clearly, $D$ is weakly Runge and further it suffices to use the Hopf lemma in the unit disc: if $u$ is a negative subharmonic function on $\DD$, then $u(\la)\leq-\gamma(1-|\la|)$, $\la\in\DD$, for some constant $\gamma>0$. 

Indeed, the function $-\dist(\cdot,\pa D)$ is convex on $D$, therefore any analytic disc $f:\DD\lon D$ satisfies $-\dist(f(\la),\pa D)\leq-\gamma(1-|\la|)$ ($\gamma$ depends on $f$).
\end{proof}
 
Recall that a domain $\Omega\su\RR^m$ is said to be \emph{strictly convex} if $$a,b\in\ov\Omega,\ a\neq b,\ t\in(0,1)\Longrightarrow ta+(1-t)b\in\Omega.$$ Note that a bounded domain $\Omega\su\RR^m$ is strictly convex if and only if $$a,b,\frac 12(a+b)\in\pa\Omega\Longrightarrow a=b.$$

\begin{cor}[cf. \cite{jp}, Proposition 11.3.3]\label{21} 
Let $D\su\CC^n$ be a bounded strictly convex domain and let $f,g:\DD\lon D$ be weak $m$-extremals for $\la_1,\ldots,\la_m$. Assume that $f(\la_j)=g(\la_j)$, $j=1,\ldots,m$. Then $f=g$.
\end{cor}

\begin{proof}
The map $h:=\frac 12(f+g):\DD\lon D$ is a weak $m$-extremal for $\la_1,\ldots,\la_m$, whence $h$ is almost proper. As $h^*=\frac 12(f^*+g^*)$ almost everywhere on $\TT$, it follows that $f^*=g^*$ almost everywhere and $f=g$.
\end{proof}

\begin{rem}
In case of the ball we can get Corollary \ref{21} by induction. In fact, for $m=2$ it is the classical result. Step $m\Longrightarrow m+1$: one may assume that $\la_{m+1}=0$ and $f(0)=g(0)=0$. Then $f(\la)=\la\phi(\la)$ and $g(\la)=\la\psi(\la)$, where $\phi,\psi$ are either $m$-extremals of $\BB_n$ or constants lying in $\pa\BB_n$. As $\phi(\la_j)=\psi(\la_j)$, $j=1,\ldots,m$, the claim follows.

On the other side, in any complex ellipsoid, equality on $m-1$ points does not suffice to claim that $f=g$. The examples are $m$-geodesics $f:=(B,0,\ldots,0)=:-g$, where $B$ is a~Blaschke product of degree $m-1$, having all zeros distinct.
\end{rem}

\begin{rem}
Recall that for 2-geodesics $f,g$ of a convex complex ellipsoid, the condition $f(\la_j)=g(\mu_j)$, $j=1,2$, where $\la_1,\la_2\in\DD$ are distinct and $\mu_1,\mu_2\in\DD$ are distinct, implies that $f=g\circ a$ for some $a\in\aut(\DD)$, see \cite[Proposition 16.2.2]{jp}. 

For $m\geq 3$ there is no an analogous property. Indeed, consider 3-geodesics $f(\la):=(\la m_\alpha(\la),0,\ldots,0)$ and $g(\la):=(\la m_\beta(\la),0,\ldots,0)$, where $\alpha,\beta\in\DD$, $\alpha\neq\beta,-\beta$. Then for any $\la\in\DD$ there is $\mu\in\DD$ such that $f(\la)=g(\mu)$, however there is no $a\in\aut(\DD)$ satisfying $f=g\circ a$ (clearly, the mappings $f$ and $g$~are not equivalent in case of the ball).

More generally, for any finite non-constant Blaschke products $B,\wi B$ there are infinite sets of different $\la$'s and $\mu$'s with $(B(\la),0,\ldots,0)=(\wi B(\mu),0,\ldots,0)$. Although, it may happen that there is no Blaschke product $B_1$ with $B=\wi B\circ B_1$ or $\wi B=B\circ B_1$, e.g. if $\deg B$ does not divide $\deg\wi B$ and vice versa (moreover, $(B,0,\ldots,0)$ and $(\wi B,0,\ldots,0)$ are not equivalent in the ball).
\end{rem}

\bigskip

We pass to problems concerning properness.

\begin{rem}
\begin{enumerate}[$(a)$]
\item Any weak $m$-extremal of a non-simply connected taut planar domain is neither proper nor almost proper. It follows from Proposition \ref{27}, infiniteness of the covering and the identity principle.
\item Any $m$-geodesic is proper.
\end{enumerate}
\end{rem}

We do not know whether any $m$-extremal is (almost) proper (P\ref{pr}).

\bigskip

Natural is the question about behavior of (weak) $m$-extremals and $m$-geodesics under compositions with proper holomorphic maps (with both sides). The problem trivializes in two cases. Indeed, if $f$ is an $m$-geodesic and $B$ is a finite non-constant Blaschke product, then $f\circ B$ is some $k$-geodesic. Note also that the mapping $$\CC\sm\{0,1\}\ni\la\longmapsto\frac{1}{\la(\la-1)}\in\CC_*$$ is proper, but $\CC\sm\{0,1\}$ has weak $m$-extremals, whereas $\CC_*$ not.

We have two simple results (cf. (P\ref{fB}) and (P\ref{inv})).

\begin{prop}
Let $D\su\CC^n$ be a convex domain and let $f:\DD\lon D$ be an $m$-extremal. Assume that $B$ is a Blaschke product of degree $k\in\NN$. Then $f\circ B:\DD\lon D$ is a weak $mk$-extremal.
\end{prop}

\begin{proof}
Let $M:=\{\la\in\DD:B'(\la)=0\}$ and let $\mu_1,\ldots,\mu_m\in\DD\sm B(M)$ be different. We will show that $f\circ B$ is a weak $mk$-extremal for elements of the set $\Lambda:=B^{-1}(\{\mu_1,\ldots,\mu_m\})$ (the structure of proper holomorphic mappings is used, cf. \cite{bed} and \cite[Chapter 15]{Rud}). Suppose that there exists $h\in\OO(\DD,D)$ such that $h(\la)=f(B(\la))$, $\la\in\Lambda$, and $h(\DD)\su\su D$. For any $\mu\in\DD\sm B(M)$ let $B_{\mu,1},\ldots,B_{\mu,k}$ denote the local inverses of $B$ in a neighborhood $U_\mu$ of $\mu$. Then $$\frac 1k (h\circ B_{\mu,1}+\ldots+h\circ B_{\mu,k})=\frac 1k(h\circ B_{\nu,1}+\ldots+h\circ B_{\nu,k})\text{\, on\, }U_\mu\cap U_\nu$$ for $\mu,\nu\in\DD\sm B(M)$. We glue these mappings to $g\in\OO(\DD\sm B(M),D)$. Then $g(\mu_j)=f(\mu_j)$ for any $j$ and $g(\DD\sm B(M))\su\su D$. Clearly, $g$ extends holomorphically to $\DD$ and the extension has a relatively compact image, contradiction.
\end{proof}

\begin{rem} 
The property of being some (weak) $m$-extremal (resp. $m$-geodesic) is not invariant under proper holomorphic mappings in different dimensions. Indeed, there exists a function $u$ harmonic in $\DD$, continuous to the boundary and such that its harmonic conjugate $v$ is not continuous on $\CDD$. We give an example from \cite[p.~253]{zy} $$u(e^{it}):=\sum_{j=2}^\infty\frac{\sin jt}{j\log j},\quad t\in\RR.$$ Adding a constant, we can assume that $u<0$ in $\CDD$. Define $\wi u:=1/2\log(1-e^{2u})$ on $\TT$, extend it harmonically to $\DD$ and take $\wi v$ as its harmonic conjugate. The map $\Phi:=(e^{u+iv},e^{\wi u+i\wi v}):\DD\lon\BB_2$ is proper, but $\Phi\circ\id_\DD$ does not extend to $\CDD$, so it is not any weak $m$-extremal of $\BB_2$.
\end{rem}

\bigskip

Following the proof of \cite[Proposition 9]{ek} we get the last result.

\begin{prop}[cf. \cite{ek}, Proposition 9] 
Let $D\su\CC^n$ be a domain and let $f:\DD\lon D$ be a holomorphic mapping such that for some $\gamma>0$ we have \begin{equation}\label{20}\dist(f(\la),\pa D)\ge\gamma(1-|\la|),\quad\la\in\DD.\end{equation} Assume that $f$ is a weak $m$-extremal for $\la_1,\ldots,\la_m$. Then $f'(\la_j)\neq 0$ for at least two $j$'s.
\end{prop}

\begin{proof}
Suppose contrary, say $f'(\la_j)=0$, $j=1,\ldots,m-1$. Then $g:=f\circ m_{-\la_{m}}$ is a weak $m$-extremal for some $\mu_1,\ldots,\mu_{m-1},0$ and $g'(\mu_j)=0$ for any $1\leq j\leq m-1$. Moreover, condition \eqref{20} for $g$ holds with possibly another constant.

For $t\in(0,1)$ consider the mapping $$h_t(\la):=g(t\la)+\sum_{j=1}^{m-1}\left(\frac{\la}{\mu_j}\prod_{k\neq j}\frac{\la-\mu_k}{\mu_j-\mu_k}\right)(g(\mu_j)-g(t\mu_j)),\quad\la\in\DD.$$ Then $h_t$ interpolates $g$ at $\mu_1,\ldots,\mu_{m-1},0$ and $\|\psi_t\|_{\DD}\to 0$ as $t\to 1$, where $$\psi_t(\la):=\frac{h_t(\la)-g(t\la)}{1-t}.$$ Hence, for $t$ sufficiently close to 1 we have $h_t(\DD)\su\su D$.
\end{proof}

\section{List of problems}
\begin{enumerate}[(P1)]
\item\label{2e2g} Does there exist a 2-extremal, which is not a 2-geodesic?
\item\label{mk} Does there exist an $m$-extremal being not any $k$-geodesic?
\item\label{m+1} Let $D\su\CC^n$ be a $k$-balanced pseudoconvex domain and let $f:\DD\lon D$ be an $m$-extremal. Assume that $k_1,\ldots,k_n\leq 1$. Decide whether the mapping $\psi(\la):=(\la^{k_1}f_1(\la),\ldots,\la^{k_n}f_n(\la))$ is an $(m+1)$-extremal.
\item\label{37} Let $f:\DD\lon\ee(p)$ be a 4-geodesic such that $f(\la)=\la\phi(\la)$, $\phi\in\OO(\DD,\ee(p))$. Does it follow that $\phi$ is a 3-geodesic?
\item\label{33} Is any weak $m$-extremal of a convex domain an $m$-extremal?
\item\label{35} Decide whether any map of the form \eqref{22} is some (weak) $l$-extremal or $l$-geodesic.
\item\label{ep} Does weak $m$-extremality coincide with $m$-extremality in any convex complex ellipsoid?
\item\label{26} Decide whether any non-constant map $(a_1B_1,\ldots,a_nB_n)$ ($a\in\pa\ee(p)$, $B_j$'s finite Blaschke products) is some (weak) $m$-extremal or $m$-geodesic.
\item\label{25} Is any $m$-extremal of $\BB_n$ equivalent with some $(a_1B_1,\ldots,a_nB_n)$?
\item\label{mkball} Is any $m$-extremal of $\BB_n$ some $k$-geodesic?
\item\label{b} Let $0<a<1$. Does it follow that the mapping $f(\la):=(a\la,\sqrt{1-a^2}\la^m)$ is an $(m+1)$-geodesic of $\BB_2$?
\item\label{pr} Decide whether any $m$-extremal is (almost) proper.
\item\label{fB} Let $f$ be a (weak) $m$-extremal and $B$ a finite non-constant Blaschke product. Does it follow that $f\circ B$ is some (weak) $k$-extremal?
\item\label{inv} Is the property of being some $m$-extremal (resp. $m$-geodesic) invariant under proper holomorphic mappings in the same dimension? 
\end{enumerate}

\bigskip

\textsc{Acknowledgements.} I would like to thank \L. Kosi\'nski and W. Zwonek. Their help, especially their ideas, had great impact on the work.

\end{document}